\documentclass[smallextended,envcountsect]{svjour3} 
\smartqed 

\usepackage{graphicx}
\usepackage{algorithm,algpseudocode}
\usepackage{mathrsfs}
\usepackage{amsmath}  
\usepackage{mathtools}
\usepackage{amssymb}
\usepackage{amsfonts} 
\usepackage{threeparttable}
\usepackage{cite}
\usepackage{url}
\usepackage{color}
\allowdisplaybreaks[1]

\newtheorem{assumption}{Assumption}

\DeclarePairedDelimiterX\Set[2]{\lbrace}{\rbrace}%
 { #1 \,\delimsize| \,\mathopen{} #2 }
\newcommand{\bs}[1]{\boldsymbol{#1}}
\newcommand{\E}{\mathbb{E}}
\newcommand{\mK}{\mathcal{K}}
\newcommand{\mL}{\mathcal{L}}

\newcommand{\mY}{\mathcal{Y}}
\newcommand{\mP}{\mathcal{P}}
\newcommand{\mZ}{\mathcal{Z}}
\newcommand{\mB}{\mathcal{B}}

\newcommand{\mI}{\mathcal{I}}

\newcommand{\conv}{\textrm{conv}}
\newcommand{\braket}[2]{\langle #1, #2 \rangle}

\newcommand{\whb}[1]{\widehat{\bs{#1}}}

\newcommand{\hbeta}{\hat{\bs{\beta}}}

\newcommand{\bvoij}{\bs{v}^{\omega ij}}
\newcommand{\bvij}[1]{\bs{v}^{\omega ij #1}}

\newcommand{\Uoij}{U^{\omega ij}}
\newcommand{\Uij}[1]{U^{\omega ij #1}}

\newcommand{\soij}{s^{\omega ij}}

\begin{document}

\title{A Decomposition Method for Distributionally-Robust Two-stage Stochastic Mixed-integer Cone Programs}

\author{Fengqiao Luo  \and  Sanjay Mehrotra }

\institute{Fengqiao Luo \at
              Department of Industrial Engineering and Management Science, Northwestern University \\
              Evanston, Illinois\\
              fengqiaoluo2014@u.northwestern.edu  
           \and
              Sanjay Mehrotra,  Corresponding author  \at
              Department of Industrial Engineering and Management Science, Northwestern University \\
              Evanston, Illinois\\
              mehrotra@northwestern.edu
}

\date{Received: date / Accepted: date}

\maketitle

\begin{abstract}
We develop a decomposition algorithm for distributionally-robust two-stage stochastic mixed-integer convex cone programs, and its important special case of distributionally-robust two-stage stochastic mixed-integer second order cone programs.  This generalizes the algorithm proposed by Sen and Sherali~[Mathematical Programming 106(2): 203-223, 2006]. We show that the proposed algorithm is finitely convergent if  the second-stage problems are solved to optimality at incumbent first stage solutions, and solution to an optimization problem to identify worst-case probability distribution is available. The second stage problems can be solved using a branch-and-cut algorithm. The decomposition algorithm is illustrated with an example. Computational results on a stochastic programming  generalization of a facility location problem show significant solution time improvements from the proposed approach. Solutions for many models that are intractable for an extensive form formulation become possible. Computational results suggest that solution time requirement does not increase significantly when considering distributional robust counterparts to the stochastic programming models.
\end{abstract}

\section{Introduction}
\label{sec:introd}
We consider the following distributionally robust two-stage stochastic mixed-integer cone program:
\begin{equation}\label{opt:DR-TSS-MICP}
\begin{aligned}
&\underset{y}{\textrm{min}}\;c^{\top}y+\underset{P\in\mP}{\textrm{max}}\;\E_P[Q(y,\omega)]  \\
&\textrm{ s.t. } F y\ge a,             \\
&\qquad y\in\mK_1\cap \{0,1\}^n,
\end{aligned}
\tag{DR-TSS-MICP}
\end{equation}
where $y$ are the first-stage decision variables that are pure-binary, and $\mK_1\subseteq\mathbb{R}^n$ is a convex cone in the Euclidian space. Here $\mP$ is the ambiguity set of probability distributions.  The ambiguity set $\mathcal{P}$ is allowed to be any convex set defined using a finite support $\Omega$ of scenarios. The recourse function $Q(y,\omega)$ is given by the following second-stage mixed-integer cone programs:
\begin{equation}\label{opt:micp-subprob}
\begin{aligned}
Q(y,\omega)=\;&\underset{x^{\omega}}{\textrm{min}}\; q^{\omega\top}x^{\omega} \\
&\textrm{ s.t. } W^{\omega}x^{\omega} \ge r^{\omega}-T^{\omega}y^k, \\
&\qquad x^{\omega}_{i}\in[z^{L\omega}_i,\;z^{U\omega}_i] & \forall i\in[l_1+l_2], \\
&\qquad x^{\omega}\in \mK_2,\;x^{\omega}\in\mathbb{Z}^{l_1}\times\mathbb{R}^{l_2}, 
\end{aligned}
\end{equation}
where $\mK_2\subseteq\mathbb{R}^{l_1+l_2}$ is a convex cone in Euclidian space. The constants $z^{L\omega}_i$ and $z^{U\omega}_i$
are lower and upper bounds on $x^{\omega}_i$ at scenario $\omega$. We let $\mathcal{B} := \{0,1\}^n$. An important special case of \eqref{opt:micp-subprob} is the distributionally-robust two-stage stochastic mixed-integer second-order-cone program:
\begin{equation}\label{opt:DR-TSS-MISOCP}
\begin{aligned}
&\underset{y}{\textrm{min}}\;\;c^{\top}y + \underset{P\in\mathcal{P}}{\textrm{max}}\;\E_P[Q(y,\omega)] \\
&\textrm{ s.t. }\;\; y\in \mY\cap\mathcal{B}.  
\end{aligned}
\tag{DR-TSS-MISOCP}
\end{equation}
For \eqref{opt:DR-TSS-MISOCP} the set $\mY$ is defined by linear and second-order-cone constraints as follows:
\begin{equation}
\mY=\Set*{y\in\mathbb{R}^n}{
\begin{array}{l}
Fy\ge a,\; \|f^{\top}_iy+g_i\|_2\le h^{\top}_iy+e_i\;\;\forall i\in[m_1],\\
0\le y_j\le 1\;\;\forall j\in[n]
\end{array} }.
\end{equation}
The recourse function $Q(y,\omega)$ in case of \eqref{opt:DR-TSS-MISOCP} is defined by the following mixed-integer second-order-cone program (MISOCP):
\begin{equation}\label{opt:subprob}
\begin{aligned}
Q(y,\omega)=\;&\underset{x^{\omega}}{\textrm{min}}\;\; q^{\omega\top}x^{\omega} & \\
&\textrm{ s.t. }\; \|A^{\omega}_ix^{\omega}+B^{\omega}_iy+b^{\omega}_i\|_2\le g^{\omega\top}_ix^{\omega}+d^{\omega}_i & \forall i\in[m], \\
&\qquad\;  W^{\omega}x^{\omega} \ge r^{\omega}-T^{\omega}y,\;\; x^{\omega}\in \mathbb{Z}^{l_1}\times\mathbb{R}^{l_2}, &  \\
&\qquad\; x^{\omega}_{i}\in[z^{L\omega}_i,\;z^{U\omega}_i] & \forall i\in[l_1+l_2], 
\end{aligned}
\end{equation} 
where $x^{\omega}$ are the second-stage decision variables for scenario $\omega$ with $l_1$ integral and $l_2$ continuous variables, 
$T^{\omega}$ is the technology matrix, and $W^{\omega}$ is the recourse matrix corresponding to scenario $\omega$. 
The algorithm developed in this paper was motivated from a distributionally-robust generalization of a service center location problem with decision-dependent customer utilities studied in our recent work \cite{luo2019_DRO-FL}. This model will be considered in our computational study. However, the method developed in this paper may have wider applicability for problems that admit \eqref{opt:DR-TSS-MISOCP} or \eqref{opt:DR-TSS-MICP} formulations. The \eqref{opt:DR-TSS-MISOCP} and \eqref{opt:DR-TSS-MICP} models have the following preliminary reformulation:
\begin{equation}\label{opt:DR-TSS-MISOCP-reform}
\begin{aligned}
&\textrm{min}\;\;c^{\top}y+\underset{P \in\mathcal{P}}{\textrm{max}}\sum_{\omega\in\Omega}p_{\omega}\eta^\omega & \\
&\textrm{ s.t. }\;\; \eta^\omega\ge Q(x,\omega) &\quad \forall \omega\in\Omega, \\
&\qquad\;\; y\in \mY,\;\eta^\omega\in\mathbb{R} &\quad \forall \omega\in\Omega.
\end{aligned}
\end{equation}

\subsection{Literature Review}
\label{sec:liter-rev}
Numerical methods have been developed for solving two-stage stochastic mixed-integer linear (TS-SMIP) and two-stage stochastic second-order-conic programs (TS-SMISOCP)  over the years. For  the two-stage stochastic linear programs (TS-SLP) with continuous second-stage variables, the recourse objective function is a piecewise-linear convex function of the first-stage variables. Benders' decomposition technique has been used to generate a piecewise-linear approximation of the convex recourse function  for this problem \cite{birge1997-introd-stoch-prog,shapiro2003-stoch-prog}.  The technique decomposes the problem as a master problem (a linear program) and multiple scenario problems (linear programs) at each iteration. Optimality and feasibility cuts derived based on the solution of the scenario problems are added to the master problem \cite{bertsimas1997_introd-linear-prog}. 

For TS-SMIP with mixed-integer second-stage variables, the objective function of the recourse problem is not convex  \cite{blair1982_val-func-int-prog,schultz1993_cont-obj-stoch-mip}. Consequently, solving TS-SMIP becomes much more challenging. Some early methods were proposed in \cite{caroe1997_cut-plane-smip,caroe1997_L-shaped-smip,louveaux1993_L-shaped-methd-smip} for solving TS-SMIP with mixed-integer second-stage variables. For example, Car{\o}e and Tind~\cite{caroe1997_cut-plane-smip,caroe1997_L-shaped-smip} suggest a conceptual algorithm based on integer programming duality. The decomposition concept is similar to the one used in  Benders' decomposition algorithm, but the algorithm uses mixed integer price functions. The master problem in this method results in nonlinear constraints, and becomes computationally challenging. Laporte and Louveaux~\cite{louveaux1993_L-shaped-methd-smip} suggested the use of a branch-and-cut algorithm that adds valid cuts to approximate the non-convex recourse function, but this algorithm is not guaranteed to find an optimal solution. In the recent years, more systematic methods are proposed for solving two-stage stochastic mixed-integer programs with mixed-integer second-stage variables. These methods are based on two main ideas. The first idea uses parametric Gomory cuts that sequentially convexify the feasible set 
\cite{sen2014_decomp-alg-gomory-cuts-ts-smip,kucukyavuz2014_finite-convergent-decomp-alg-ts-smip}. 
Gade et al. \cite{sen2014_decomp-alg-gomory-cuts-ts-smip} show finite-convergence of their algorithm for solving TS-SMIPs with pure-binary first-stage variables and pure-integer second-stage variables based on generating Gomory cuts 
that are parameterized by the first-stage solution. 
This approach is generalized by Zhang and K{\"u}{\c c}{\"u}kyavuz~\cite{kucukyavuz2014_finite-convergent-decomp-alg-ts-smip}
for solving TS-SMIP with pure-integer variables in both stages.  Recent work has also provided insights into developing tighter formulations by identifying globally valid parametric inequalities (see \cite{bansal2018_tight-ts-smip}, and references therein). Specifically, Bansal and Zhang~\cite{bansal2018_dr-ts-mixed-int-p-conic-prog} have developed 
nonlinear sparse cuts for tightening the second-stage formulation of a class of two-stage stochastic $p$-order conic mixed-integer programs by extending the results of \cite{atamturk2008_conic-mip-cuts} on convexifying a simple polyhedral conic mixed integer set.

The second idea is based on  solving a mixed-integer linear program in the second stage using a branch-and-cut technique \cite{sen2006_decomp-BB-TSS-MIP}. This idea has merit because in practice algorithms that are only based on cutting planes may not be efficient in finding an optimal solution of a mixed binary or mixed integer program. The algorithm developed by Sen and Sherali~\cite{sen2006_decomp-BB-TSS-MIP} allows for using the branch-and-cut method when solving the second stage scenario problems. Sen and Sherali~\cite{sen2006_decomp-BB-TSS-MIP} showed finite convergence of their algorithm for linear problems with pure-binary first-stage variables and mixed-binary second-stage variables. The decomposition branch-and-cut algorithm of \cite{sen2006_decomp-BB-TSS-MIP} uses 
the branch-and-cut method in solving the second-stage problems either completely or partially for given first stage solutions. A union of the sets represented by the leaf nodes generated during the branch-and-cut process is taken to generate a valid optimality constraint at the current first-stage solution. A disjunctive programming formulation is used for the generation of this optimality constraint \cite{balas1985,balas1998,sen2005_set-convf-ts-smip,sen1985_convg-cut-plane-nonconv-math-prog}.  The decomposition branch-and-cut algorithm is enhanced in \cite{sen2017_ancestral-bender-cut-ts-smip}
to solve TS-SMIP with mixed-integer variables in both stages. A recent tutorial on two-stage stochastic mixed-integer programming can be found in K{\"u}{\c c}{\"u}kyavuz and Sen~\cite{kucukyavuz2017-introd-ts-stoch-mip}.

The Benders' decomposition algorithms developed for two-stage stochastic mixed integer programs with binary first stage variables can be modified to solve distributionally-robust two-stage stochastic mixed integer programs \cite{bansal2018-decomp-alg-distr-rob-mip}, and distributionally robust two-stage stochastic disjunctive programs \cite{bansal2019_ts-dr-disj-prog}. The algorithms retain their finite-convergence property if  the ambiguity set is polyhedral representable \cite{bansal2018-decomp-alg-distr-rob-mip}, or more generally when a separation oracle for the unknown probability distribution is available \cite{bansal2019_ts-dr-disj-prog}. 

The work in this paper is most closely related to \cite{sen2006_decomp-BB-TSS-MIP}, while it benefits from the observations in \cite{bansal2018-decomp-alg-distr-rob-mip,bansal2019_ts-dr-disj-prog}.  This paper makes the following contributions. It shows that the algorithmic framework of Sen and Sherali~\cite{sen2006_decomp-BB-TSS-MIP} can be generalized to solve two-stage stochastic mixed integer second order cone and convex conic programs, and their distributionally robust counterparts \eqref{opt:DR-TSS-MISOCP}, \eqref{opt:DR-TSS-MICP} to optimality in a finite number of iterations. The paper also discusses the computational performance of the proposed algorithm in the context of a facility location problem, which is formulated as a distributionally robust two-stage mixed integer second order cone program.  The results show that the solution times of the decomposition algorithm are significantly better than those required to solve an extended formulation of the problem.

The paper is organized as follows. In Section~\ref{sec:alg} our algorithm for solving \eqref{opt:DR-TSS-MISOCP} is developed. The generalization of this algorithm for \eqref{opt:DR-TSS-MICP} is given in Section~\ref{sec:DR-TSS-MICP}. An example illustrating the algorithm is given in Section~\ref{sec:smallExample}.
Section~\ref{sec:num-exp} discusses computational results on instances of a facility location model. This is followed by some concluding remarks in Section~\ref{sec:conclusions}.

\section{An Algorithm for Solving Distributionally Robust Two-stage Stochastic Mixed-integer Cone Programs}
\label{sec:alg}

\subsection{Structure of the algorithm}
\label{sec:struct}
The algorithm presented in this section consists of solving a pure-binary SOCP first-stage problem 
and mixed-integer SOCP second-stage problems corresponding to each scenario at each iteration.
The central idea is to generate an aggregated optimality constraint that is valid
for the first-stage problem. The first-stage problem is formulated as:
\begin{equation}\label{opt:masterprob}
\begin{aligned}
&\underset{x,\eta}{\textrm{min}}\;\; c^{\top}y + \eta &\\
&\textrm{ s.t. }\; \eta \ge h^l - (f^l)^{\top}y  & \forall  l\in \{1,\ldots,k-1\}, \\
&\qquad\; y\in \mY\cap\mB,\;\eta\in\mathbb{R},
\end{aligned}
\end{equation}
where  $k-1$ is the number of iterations that have been completed. 
The constraint $\eta \ge h^l- (f^l)^{\top}y$ (defined later) is an aggregated optimality constraint which is obtained from a risk-averse aggregation 
of scenario based optimality constraints at iteration $l\in[k-1]$, where $h^l$ and $f^l$ are appropriate coefficients. The approach to generate this constraint is as follows:
\begin{enumerate}
	\item Solve the current first-stage problem.
	\item Given the current first-stage solution, solve second-stage problems using a branch-and-cut algorithm.  
	\item For each scenario and each leaf node of the branch-and-cut tree for the scenario, 
		  generate an optimality constraint for a scenario-node (Section~\ref{sec:sc-nd-cut}). 
	\item For each scenario, generate a valid optimality constraint
		  by taking the union of the epigraph defined by the scenario-node  
		  optimality constraints and using a lift-and-project technique (Section~\ref{sec:sc-cut}).		  
	\item Generate an aggregated optimality constraint by aggregating over the worst-case probability distribution of the scenarios,
		  and add it to the first-stage problem (Section~\ref{sec:agg-cut}). 
	\item Repeat Step 1.
\end{enumerate}

\subsection{Scenario-node based optimality constraint}
\label{sec:sc-nd-cut}
In the $k^{\textrm{th}}$ (main) iteration, 
we solve the first-stage problem to optimality and obtain the current optimal solution $y^k$
of the first-stage problem \eqref{opt:masterprob}. Then we substitute the first-stage solution $y^k$ into the second-stage problem 
\eqref{opt:subprob} for each $\omega\in\Omega$. The second-stage problem is solved partially (not necessarily to optimality) using a branch-and-cut method. 
We use $\textrm{Sub}(y^k,\omega)$ to denote the second-stage problem \eqref{opt:subprob}
for the first-stage solution $y^k$ and scenario $\omega$. After partially solving $\textrm{Sub}(y^k,\omega)$
using the branch-and-cut method,
we obtain a branch-and-cut tree corresponding to this sub-problem. Let $\mathcal{L}(y^k,\omega)$ be the set of leaf nodes of the tree. 
Consider a node $v\in \mathcal{L}(y^k,\omega)$ and let the SOCP relaxation associated with the node be given by:
\begin{equation}\label{opt:SOCP-node}
\begin{aligned}
&\underset{x}{\textrm{min}}\;\; q^{\omega\top}x^{\omega} & & \quad \textrm{dual multipliers}\\
&{\rm s.t. }\; \|A^{\omega}_ix^{\omega}+B^{\omega}_iy^k+b^{\omega}_i\|_2\le g^{\omega\top}_ix^{\omega}+d^{\omega}_i & \forall i\in[m], & \qquad [\lambda^{k\omega v}_i,\;\theta^{k\omega v}_i]\in\textrm{SOC} \\
&\qquad\; W^{\omega}x^{\omega} \ge r^{\omega}-T^{\omega}y^k, & & \qquad \gamma^{k\omega v}_1 \\
&\qquad\; X^{\omega}_vx^{\omega}\ge t^{\omega}_v, & & \qquad  \gamma^{k\omega v}_2  \\
&\qquad\; x^{\omega}\ge {z}^{L\omega}_{v}, \;\; x^{\omega}\le z^{U\omega}_{v}, & & \qquad \tau^{k\omega v}_L, \;\; \tau^{k\omega v}_U
\end{aligned}
\end{equation}  
where $z^{L\omega}_v$ and $z^{U\omega}_v$ are the lower and upper bound vectors of $x^{\omega}$ associated with 
node $v$. For the continuous entries of $x^{\omega}$, the corresponding entries of $z^{L\omega}_v$ and $z^{U\omega}_v$
are given by $z^{L\omega}_{v,i}:=z^{L\omega}_i$ and $z^{U\omega}_{v,i}:=z^{U\omega}_i$ for all $i\in[l_1]$.
If $z^{L\omega}_v$ and $z^{U\omega}_v$ corresponding to the integral entries of $x$ in the node $v$ are equal, the corresponding variables are fixed.
The constraints $X^{\omega}_vx^{\omega}\ge t^{\omega}_v$ are additional structure-based cuts used to tighten the second-stage reformulation
that is associated with node $v$. Note that if no such cuts are generated for 
tightening the second-stage reformulation, the constraints $X^{\omega}_vx^{\omega}\ge t^{\omega}_v$ are not present in \eqref{opt:SOCP-node}.
Let $[\lambda^{k\omega v}_i,\;\theta^{k\omega v}_i]$ for $i\in[m]$, $\gamma^{k\omega v}_1$, $\gamma^{k\omega v}_2$, $\tau^{k\omega v}_L$ and $\tau^{k\omega v}_U$  be 
the dual multipliers associated with the SOCP constraint, the recourse constraint, the additional constraints,
and the lower and upper bounding constraints, respectively. 
We make the following assumptions when considering \eqref{opt:DR-TSS-MISOCP}.
\begin{assumption}\label{ass:feasible}
\begin{itemize}
	\item[\emph{(a)}] The  \eqref{opt:DR-TSS-MISOCP} problem has a complete recourse.
	\item[\emph{(b)}] For any $x^1\in\mathbb{Z}^{l_1}$ satisfying $x^1_i\in[z^{L\omega}_i,\;z^{U\omega}_i]$, 
			   there exists $x^2\in\mathbb{R}^{l_2}$
			   such that the solution $x=[x^1;x^2]$ is strictly feasible to the second-stage problem \eqref{opt:subprob}.
\end{itemize}
\end{assumption}
Assumption~\ref{ass:feasible}(a) ensures that every first-stage feasible solution leads to a feasible second-stage problem
for all scenarios. Assumption~\ref{ass:feasible}(b) ensures that all the feasible sets associated with relaxation of the terminal nodes  
form a partition of the feasible set of the second-stage problem at each scenario, and strong duality holds for every node relaxation SOCP (see Appendix~\ref{app:strong-dual}). We note that if Assumption~\ref{ass:feasible} is violated, we can introduce artificial variables in the instance to make the assumption satisfied. 
The conic dual of \eqref{opt:SOCP-node} can be formulated as
\begin{equation}\label{opt:SOCP-node-dual}
\begin{aligned}
&\textrm{max}\;\;\left[\sum^m_{i=1}(\theta^{k\omega v}_i )^{\top}B^{\omega}_i - (\gamma^{k\omega v}_1)^{\top}T^{\omega}\right] y^k
+\sum^m_{i=1}\left(b^{\omega\top}_i\theta^{k\omega v}_i-d^{\omega}_i\lambda^{k\omega v}_i\right) \\
&\qquad +r^{\omega\top}\gamma^{k\omega v}_1+t^{\top}_v\gamma^{k\omega v}_2 
+z^{L\omega\top}_v\tau^{k\omega v}_L-z^{U\omega\top}_v\tau^{k\omega v}_U    \\
&\textrm{ s.t. }\;\; \sum^m_{i=1}(g^{\omega}_i\lambda^{k\omega v}_i-A^{\omega\top}_i\theta^{k\omega v}_i)+W^{\omega\top}\gamma^{k\omega v}_1+X^{\omega\top}_v\gamma^{k\omega v}_1
+\tau^{k\omega v}_L-\tau^{k\omega v}_U=q^{\omega},  \\
&\qquad\;\; \|\theta^{k\omega v}_i\|_2\le \lambda^{k\omega v}_i \qquad\quad i\in[m],   \\
&\qquad\;\; \gamma^{k\omega v}_1\ge 0,\;\gamma^{k\omega v}_2\ge 0,\;\tau^{k\omega v}_L\ge 0, \; \tau^{k\omega v}_U\ge 0,
\; \lambda^{k\omega v}_i\ge 0\;\forall i\in[m].
\end{aligned}
\end{equation}

\begin{proposition}\label{prop:dual-ineq}
Let Assumption~\ref{ass:feasible} be satisfied.
Let $y^*$ be the optimal solution of \eqref{opt:DR-TSS-MISOCP}, and $x^{\omega *}$ be an optimal second-stage
solution of $\textrm{Sub}(y^*,\omega)$. 
Let $[\lambda^{k\omega v}_i,\;\theta^{k\omega v}_i]$ for $i\in[m]$, $\gamma^{k\omega v}_1$, 
$\gamma^{k\omega v}_2$, $\tau^{k\omega v}_L$ and $\tau^{k\omega v}_U$
be the dual variable values determined by solving \eqref{opt:SOCP-node} corresponding to node $v$ of $\textrm{Sub}(y^k,\omega)$.
If $x^{\omega*}$ satisfies $z^{L\omega}_v\le x^{\omega*}\le z^{U\omega}_v$, then
the following inequality holds\emph{:} 
\begin{equation}\label{eqn:valid-ineq-w-v}
\begin{aligned}
Q(y^*,\omega)\ge &\left[\sum^m_{i=1}(\theta^{k\omega v}_i )^{\top}B^{\omega}_i - (\gamma^{k\omega v}_1)^{\top}T^{\omega}\right] y^*+\sum^m_{i=1}\left( b^{\omega\top}_i\theta^{k\omega v}_i-d^{\omega}_i\lambda^{k\omega v}_i\right) \\
&+r^{\omega\top}\gamma^{k\omega v}_1+t^{\top}_v\gamma^{k\omega v}_2 +z^{L\omega\top}_v\tau^{k\omega v}_L-z^{U\omega\top}_v\tau^{k\omega v}_U.
\end{aligned}
\end{equation}
\end{proposition}
\begin{proof}
Note that the expression on the right side of \eqref{eqn:valid-ineq-w-v} is the objective in 
\eqref{opt:SOCP-node-dual} by replacing $y^k$ with $y^*$. 
We denote this expression as $\psi(y^*)$ in the proof. The SOCP duality  implies that
\begin{displaymath}
\begin{aligned}
&Q(y^*,\omega)=q^{\omega\top}x^{\omega*} \\
&= \left[\sum^m_{i=1}(g_i\lambda^{k\omega v}_i-A^{\omega\top}_i\theta^{k\omega v}_i)+W^{\omega\top}\gamma^{k\omega v}_1+X^{\omega\top}_v\gamma^{k\omega v}_2
+\tau^{k\omega v}_L-\tau^{k\omega v}_U \right]^{\top}x^{\omega*}  \\
&=\sum^m_{i=1}\left[\lambda^{k\omega v}_ig^{\omega\top}_ix^{\omega *}-(\theta^{k\omega v}_i)^{\top}A^{\omega}_ix^{\omega *}\right] 
+ (\gamma^{k\omega v}_1)^{\top}W^{\omega}x^{\omega *}+(\gamma^{k\omega v}_2)^{\top}X^{\omega}_vx^{\omega *} \\
&\qquad +(\tau^{k\omega v}_L)^{\top}x^{\omega *}+(-\tau^{k\omega v}_U)^{\top}x^{\omega *} \\
&\ge \sum^m_{i=1}\left[\lambda^{k\omega v}_ig^{\omega\top}_ix^{\omega *}-(\theta^{k\omega v}_i)^{\top}A^{\omega}_ix^{\omega *}\right] + (\gamma^{k\omega v}_1)^{\top}[r^{\omega}-T^{\omega}y^k] \\
&\qquad+(\gamma^{k\omega v}_2)^{\top}t^{\omega}_v+(\tau^{k\omega v}_L)^{\top}z^{L\omega}_v -(\tau^{k\omega v}_U)^{\top}z^{U\omega}_v  \\
\end{aligned}
\end{displaymath}
\begin{displaymath}
\begin{aligned}
&=\psi(y^*) + \sum^m_{i=1}\left[\lambda^{k\omega v}_ig^{\omega\top}_ix^{\omega *}-(\theta^{k\omega v}_i)^{\top}A^{\omega}_ix^{\omega *}\right] \\
&\qquad-\sum^m_{i=1}[(b^{\omega}_i+B^{\omega}_iy^*)^{\top}\theta^{k\omega v}_i-d^{\omega}_i\lambda^{k\omega v}_i] \\
&=\psi(y^*)+\sum^m_{i=1}\left[\lambda^{k\omega v}_i(g^{\omega\top}_ix^{\omega *}+d^{\omega}_i) - (\theta^{k\omega v}_i)^{\top}(A^{\omega}_ix^{\omega *}+B^{\omega}_iy^*+b^{\omega}_i)\right] \\
&\ge \psi(y^*),
\end{aligned}
\end{displaymath}
where we have applied the inequalities given by constraints in \eqref{opt:SOCP-node} and \eqref{opt:SOCP-node-dual},
the non-negativity of dual variables and the inner product in the second-order cone to derive the above inequality.
\end{proof}
Note that inequality \eqref{eqn:valid-ineq-w-v} is a scenario-node based optimality constraint.
The validity of \eqref{eqn:valid-ineq-w-v} depends on whether the optimal second-stage
solution $x^{\omega*}$ corresponding to the optimal solution $y^*$ of \eqref{opt:DR-TSS-MISOCP}
satisfies $z^{L\omega}_v\le x^{\omega*}\le z^{U\omega}_v$, where $z^{L\omega}_v$ and $z^{U\omega}_v$
are the lower and upper bounds  that specify node $v$ of $\textrm{Sub}(y^k,\omega)$.

\subsection{Scenario based optimality constraint}
\label{sec:sc-cut}
It is clear that the second-stage optimal solution $x^{\omega *}$ is in the set:
\begin{displaymath}
\mZ^{\omega}=\Set*{x\in\mathbb{R}^{l_1+l_2}}{ z^{L\omega}_{i}\le x_i\le z^{U\omega}_{i},\; \forall i\in[l_1+l_2]}.
\end{displaymath}
The hyper-rectangle specified by node $v$ of $\textrm{Sub}(y^k,\omega)$ is 
\begin{displaymath}
\mZ^{\omega}_v=\Set*{x\in\mathbb{R}^{l_1+l_2}}{ z^{L\omega}_{v,i}\le x_i\le z^{U\omega}_{v,i},\; \forall i\in[l_1+l_2]}.
\end{displaymath}
Assumption~\ref{ass:feasible} implies that $\mZ^{\omega}=\cup_{v\in\mathcal{L}(y^k,\omega)}\mZ^{\omega}_v$
, i.e., the set of leaf nodes form a partition of $\mZ^{\omega}$. 
Therefore, the second-stage optimal solutions $x^{\omega *}$ are feasible at some node 
in $\mathcal{L}(y^k,\omega)$.
Based on this property, we can generate a valid scenario based optimality cut that is independent of the node 
of $\textrm{Sub}(y^k,\omega)$ that contains $x^{\omega *}$ using the disjunctive programming technique 
\cite{balas1998,balas1985}. 
To generate this disjunctive programming based optimality cut we consider the following epigraphs:
\begin{equation}\label{eqn:Ekvw}
E^{k\omega}_v=\Set*{(\eta^{\omega},y)\in\mathbb{R}\times\mathbb{R}^n}
{\eta^{\omega}\ge R^{k\omega v\top} y+S^{k\omega v}, \; y\in \mY^{\prime}},
\end{equation}
where $R^{k\omega v}$ and $S^{k\omega v}$ are the coefficient vector of $y^*$ and the constant in
the right-side expression of \eqref{eqn:valid-ineq-w-v}, respectively. Specifically, they are given as follows:
\begin{displaymath}
\begin{aligned}
&R^{k\omega v}=\sum^m_{i=1}(\theta^{k\omega v}_i )^{\top}B^{\omega}_i - (\gamma^{k\omega v}_1)^{\top}T^{\omega}, \\
&S^{k\omega v}=\sum^m_{i=1}\left( b^{\omega\top}_i\theta^{k\omega v}_i-d^{\omega}_i\lambda^{k\omega v}_i\right)
+r^{\omega\top}\gamma^{k\omega v}_1+t^{\top}_v\gamma^{k\omega v}_2 
+z^{L\omega\top}_v\tau^{k\omega v}_L-z^{U\omega\top}_v\tau^{k\omega v}_U.
\end{aligned}
\end{displaymath} 
The set $\mY^{\prime}$ is a polytope defined
as $\mY^{\prime}=\Set*{y\in\mathbb{R}^n}{Fy\ge a,\;0\le y_j\le 1\;\;\forall j\in[n]}$. We are interested in the union polytope 
$\Pi^{k\omega}$ defined as $\Pi^{k\omega}=\cup_{v\in\mathcal{L}(y^k,\omega)}E^{k\omega}_v$.
Proposition~\ref{prop:epigraph} gives a property of $\Pi^{k\omega}$.

\begin{proposition}\label{prop:epigraph}
Let $y$ be any feasible solution of \eqref{opt:DR-TSS-MISOCP}.
Suppose that the point $(\eta^{\omega},y)\in\mathbb{R}\times\mY$ satisfies $\eta^{\omega}\ge Q(y,\omega)$. Then $(\eta^{\omega},y)\in\Pi^{k\omega}$.
\end{proposition}
\begin{proof}
Let $x^{\omega}$ be the optimal second-stage solution of $\textrm{Sub}(y,\omega)$.
There exists a node $v\in\mathcal{L}(y,\omega)$ such that $x^{\omega}\in\mZ^{\omega}_{v}$.
Then Proposition~\ref{prop:dual-ineq} implies that
$\eta^{\omega}\ge Q(y,\omega)\ge R^{k\omega v} y+S^{k\omega v}$,
and hence $(\eta^{\omega},y)\in E^{k\omega}_{v}$.
Therefore, we have $(\eta^{\omega},y)\in\Pi^{k\omega}$.
\end{proof}

Based on Proposition~\ref{prop:epigraph}, we now construct a valid inequality 
using the lift-and-project technique \cite{balas1998,conforti-IP-2014}.
Using Theorem~2.1 of \cite{balas1998}, the convex hull of $\Pi^{k\omega}$ can be represented as:
\begin{displaymath}
\textrm{conv}(\Pi^{k\omega})=\textrm{Proj}_{y,\eta^{\omega}}\Set*{\begin{array}{l}
y,y_{v},\eta^{\omega}, \\
\eta^{\omega}_{v},\alpha_v,\\
\forall v\in 
\mathcal{L}(y^k,\omega)
\end{array}}
{\begin{aligned}
&\eta^{\omega}=\sum_{v\in\mathcal{L}(y^k,\omega)}\eta^{\omega}_v, \\
& y = \sum_{v\in\mathcal{L}(y^k,\omega)}y_v, \\
&\eta^{\omega}_v\ge R^{k\omega v}y_v+S^{k\omega v}\alpha_v, \\
& Fy_v\ge a\alpha_v,\;\; \bs{0}\le y_v\le \alpha_v\bs{1}, \\
&  \sum_{v\in\mathcal{L}(y^k,\omega)}\alpha_v = 1,\; \alpha_v \ge 0 \\
& \qquad\forall v\in\mathcal{L}(y^k,\omega)
\end{aligned}}
.
\end{displaymath} 
Based on the above representation and using Theorem~3.1 of \cite{balas1998}, 
a valid inequality of $\textrm{conv}(\Pi^{k\omega})$ can be written as:
\begin{equation}\label{eqn:cut-omega}
\eta^{\omega}\ge \lambda^{k\omega\top} y + \zeta^{k\omega},
\end{equation}
if and only if $(\lambda^{\omega},\zeta^{\omega})\in\textrm{Proj}_{\lambda,\zeta} V^{k\omega}$,
where $V^{k\omega}$ is a polyhedron defined as follows:
\begin{equation}\label{eqn:Vk}
V^{k\omega}=\Set*{\begin{aligned}&\lambda,\zeta,\sigma_{v},\gamma_v\\ 
&\forall v\in\mathcal{L}(x^k,\omega)\end{aligned}}
{
\begin{aligned}
&\lambda-R^{k\omega v}+F^{\top}\sigma_v-\gamma_v \le 0 \quad \forall v\in\mathcal{L}(x^k,\omega), \\
&\zeta-\sigma^{\top}_{v}a-S^{k\omega v}+\gamma^{\top}_v\bs{1} \le 0 \quad \forall v\in\mathcal{L}(x^k,\omega), \\
&\sigma_{v}\ge 0,\;\gamma_v\ge 0
\end{aligned}
}.
\end{equation}
The inequality \eqref{eqn:cut-omega} becomes a scenario based optimality constraint if we choose the coefficients
$(\lambda^{k\omega},\zeta^{k\omega})$ in \eqref{eqn:cut-omega} to be the optimal solution of the linear program LP($y^k,\omega$):
\begin{equation}\label{opt:disj-cut}
\begin{aligned}
&\textrm{max}\;\;\lambda^{\top}y^k+\zeta \\
&\textrm{ s.t. }\; \{\lambda,\zeta,\sigma_{v},\gamma_v\;\forall v\in\mathcal{L}(y^k,\omega)\}\in V^{k\omega}.
\end{aligned}
\end{equation}
Note that the scenario based inequality \eqref{eqn:cut-omega} does not depend on node indices. 
Proposition~\ref{prop:LP-bd} shows that the optimal value of linear program \eqref{opt:disj-cut} is bounded.
\begin{proposition}\label{prop:LP-bd}
The linear program \eqref{opt:disj-cut} is feasible. It has a finite optimal value, and  an extreme-point optimal solution.
\end{proposition}
\begin{proof}
It is easy to verify that the solution $\sigma_v=0$, $\gamma_v=\bs{0}$, \newline
$\lambda=\textrm{min}_{v\in\mL(y^k,\omega)}\;R^{k\omega v}$, 
$\zeta=\textrm{min}_{v\in\mL(y^k,\omega)}\;S^{k\omega v}$ 
is an extreme point of the polyhedron defined by \eqref{opt:disj-cut}.
Therefore, \eqref{opt:disj-cut} is feasible.
Choose an arbitrary leaf node $v\in\mathcal{L}(y^k,\omega)$. 
Based on the constraints in \eqref{eqn:Vk}, we have
$\lambda\le R^{k\omega v}-F^{\top}\sigma_v+\gamma_v$ and 
$\zeta\le S^{k\omega v}+\sigma^{\top}_va-\gamma^{\top}_v\bs{1}$.
The objective of \eqref{opt:disj-cut} satisfies:
\begin{displaymath}
\begin{aligned}
&\lambda^{\top}y^k+\xi\le (R^{k\omega v\top}-\sigma^{\top}_vF+\gamma^{\top}_v)y^k
+S^{k\omega v}+\sigma^{\top}_va-\gamma^{\top}_v\bs{1} \\
&=R^{k\omega v\top}y^k+S^{k\omega v}+\sigma^{\top}_v(a-Fy^k)+\gamma^\top_v(y_k-\bs{1})
\le R^{k\omega v\top}y^k+S^{k\omega v},
\end{aligned}
\end{displaymath}
where the last inequality is due to $\sigma_v\ge 0$, $Fy^k\ge a$, $\gamma_v\ge \bs{0}$,
and $y_k\le\bs{1}$. 
Therefore, the optimal value of \eqref{eqn:Vk} is finite.

We prove by contradiction to show that \eqref{opt:disj-cut} has an extreme-point optimal solution.
Suppose $\bs{v}^*$ is an optimal solution of  \eqref{opt:disj-cut}.  
Let $\{\bs{v}^s\}_{s\in S}$ and $\{\bs{r}^t\}_{t\in T}$ be the set of 
extreme points and extreme rays of the polyhedron defined by the feasible set
of \eqref{opt:disj-cut}. 
Then from the polyhedra decomposition theorem $\bs{v}^*$ can be represented as
\begin{displaymath}
\bs{v}^*=\sum_{s\in S}\beta_s\bs{v}^s + \sum_{t\in T}c_t\bs{r}^t,
\end{displaymath}
for some coefficients $\beta_s\ge 0$, $\sum_{s\in S}\beta_s=1$,
and $c_t\ge 0$. Note that the objective value can not increase
along any extreme ray (otherwise, the problem is unbounded).
We can take $c_t=0$. Let $\bs{v}_{s_0}$ be the extreme point
that has the largest objective value among all $s\in S$.
Then $\bs{v}_{s_0}$ is an extreme point optimal solution of \eqref{opt:disj-cut}.
\end{proof}
As a consequence of Proposition~\ref{prop:LP-bd} a valid inequality \eqref{eqn:cut-omega} is available from an optimal 
vertex solution of \eqref{opt:disj-cut}.

\subsection{Aggregated optimality constraint}
\label{sec:agg-cut}
We determine the unknown probability distribution $P$ as a worst-case (risk-averse) distribution $\bs{p}^k:=\Set*{p^k_{\omega}}{\omega\in\Omega}$
based on the current first-stage solution $y^k$. Specifically, let $\bs{p}^k$ be an optimal solution of the optimization problem:
\begin{equation}\label{opt:worst-distr}
\underset{\bs{p}\in\mathcal{P}}{\textrm{max}}\; \sum_{\omega\in\Omega}p_{\omega}\left(\lambda^{k\omega\top}y^k+\zeta^{k\omega}\right),
\end{equation}
where $y^k$ is the current first-stage solution, and $\lambda^{k\omega}$ and $\zeta^{k\omega}$ are determined by the linear program 
\eqref{opt:disj-cut}.
We aggregate inequality \eqref{eqn:cut-omega} over all scenarios using $\bs{p}^k$. The aggregated optimality cut is given by:
\begin{equation}\label{eqn:agg-cut}
\eta \ge \sum_{\omega\in\Omega} p^k_{\omega}\lambda^{k\omega\top}y + \sum_{\omega\in\Omega}p^k_{\omega}\zeta^{k\omega}.
\end{equation}
The constraint \eqref{eqn:agg-cut} is added to \eqref{opt:subprob} at the end of iteration $k$. 
\begin{proposition}\label{prop:agg-constr-valid}
The aggregated constraint \eqref{eqn:agg-cut} is valid for \eqref{opt:DR-TSS-MISOCP-reform}.
\end{proposition}
\begin{proof}
Since \eqref{eqn:cut-omega} is valid for $\conv(\Pi^{k\omega})$,  using Proposition~\ref{prop:epigraph}
we have $Q(y,\omega)\ge \lambda^{k\omega\top}y+\zeta^{k\omega}$. Then it follows that
\begin{displaymath}
\underset{P\in\mP}{\textrm{max}}\;\E_{P}[Q(y,\omega)]\ge \sum_{\omega\in\Omega}p^k_{\omega}Q(y,\omega)
\ge \sum_{\omega\in\Omega} p^k_{\omega}\lambda^{k\omega\top}y + \sum_{\omega\in\Omega}p^k_{\omega}\zeta^{k\omega}.
\end{displaymath}
Therefore, \eqref{eqn:agg-cut} is a valid constraint for \eqref{opt:DR-TSS-MISOCP-reform}.
\end{proof}

\subsection{The algorithm and its properties}
\label{sec:alg-prop}
The decomposition branch-and-cut (DBC) algorithm for solving \newline \eqref{opt:DR-TSS-MISOCP}
is given as Algorithm~\ref{alg:decomp-BC}.
This algorithm may use a partial branch-and-bound tree to generate the
aggregated optimality constraints from Sections~\ref{sec:sc-nd-cut}-\ref{sec:agg-cut}. However, ensuring finite convergence to an optimal solution requires that the scenario problems are solved to optimality. The finite convergence of Algorithm~\ref{alg:decomp-BC} is shown in Theorem~\ref{thm:finite-convg}. We need the following two intermediate results in the proof of this theorem.
\begin{proposition}\label{prop:extreme-pt}
Let $\mY^{\prime}=\Set*{y\in\mathbb{R}^n}{Fy\ge a,\;0\le y_j\le 1\;\;\forall j\in[n]}$ be a polytope.
Let $y^k$ be an optimal solution of the first-stage problem \eqref{opt:masterprob} at the $k^{\textrm{th}}$ iteration.
Then $y^k$ is an extreme point of $\mY^{\prime}$.
\end{proposition}
\begin{proof}
Denote by Master-$k$ the first-stage problem \eqref{opt:masterprob} at the $k^{\textrm{th}}$ iteration.
Let $P=\conv(\mY\cap\mB)$ be the convex hull of the feasible set of the first-stage problem.
It is easy to see that $P$ is a polytope, and $P\subseteq \mY^{\prime}$. 
Note that $y^k$ is a binary vector. We prove by contradiction that $y^k$ is an extreme point of $\mY^{\prime}$.
Suppose $y^k$ is not an extreme point of $\mY^{\prime}$.
Then there exist a $r\in\mathbb{N}_+$, a subset of extreme points $\{u^i\}^r_{i=1}$ of $\mY^{\prime}$, 
and coefficients $\{\alpha^i\}^r_{i=1}$ satisfying the following convex-combination equations:
\begin{equation}\label{eqn:conv-comb2}
y^k=\sum^r_{i=1}\alpha_iu^i,\quad \sum^r_{i=1}\alpha_i=1,\quad \alpha_i>0\;\;\forall i\in[r].
\end{equation}
We divide the discussion into two cases. Case 1: There exists an index $i\in[r]$ and an index $s\in[n]$ such that 
$0<u^i_s<1$. In this case, we must have $0<y^k<1$, which contradicts with that $y^k$ is a binary vector.
Case 2: All the points $\{u^i\}^r_{i=1}$ are integral points. For any index $l\in[n]$, 
we let $\mI^l_0=\Set*{i\in[r]}{u^i_l=0}$ and $\mI^l_1=\Set*{i\in[r]}{u^i_l=1}$.
There must exist an index $s\in[n]$, such that both $\mI^s_0$ and $\mI^s_1$ are non-empty.
Otherwise, all the points in $\{u^i\}^r_{i=1}$ are equal to each other.
Then we have 
\begin{displaymath}
y^k_s=\sum_{i\in\mI^s_0}\alpha_iu^i_s+\sum_{i\in\mI^s_1}\alpha_iu^i_s=\sum_{i\in\mI^s_1}\alpha_i,
\end{displaymath}
and hence $0<y^k_s<1$, which contradicts with the assumption that $y^k$ is a binary vector.
Therefore, $y^k$ must be an extreme point of $\mY^{\prime}$.
\end{proof}

\begin{lemma}\label{lem:Q=Ry+Z}
Let $y^{k}$ be a first-stage feasible solution at the $k^{\textrm{th}}$ iteration, and $\omega\in\Omega$ be a scenario. 
Suppose the second-stage problem \emph{Sub($y^k,\omega$)} is solved to optimality using the branch-and-cut method, 
and $\mathcal{L}(y^k,\omega)$ is the set of leaf nodes of the branch-and-cut tree when the branch-and-cut algorithm terminates. 
Then we have $Q(y^k,\omega)=\lambda^{k\omega*\top}y^k+\zeta^{k\omega*}$, 
where the coefficients $\lambda^{k\omega*}$ and $\zeta^{k\omega*}$ are the optimal solution of \eqref{opt:disj-cut}.
\end{lemma}
\begin{proof}
From Proposition~\ref{prop:extreme-pt}, $y^k$ is an extreme point of $\mY^{\prime}$.
Since the second-stage problem Sub($y^k,\omega$) is solved to optimality,
there must exist a node $v\in\mathcal{L}(y^k,\omega)$
such that the optimal solution $x^{\omega*}$ of Sub($y^k,\omega$) is 
contained in the feasible set associated with node $v^*$.
Strong duality implies that 
\begin{displaymath}
Q(y^k,\omega)=q^{\top}x^{\omega*}=R^{k\omega v^*\top}y^k+S^{k\omega v^*}.
\end{displaymath}
Therefore, the point $(Q(y^k,\omega),y^k)$ is in the epigraph 
\begin{displaymath}
E^{k\omega}_{v^*}=\Set*{(\eta^{\omega},y)\in\mathbb{R}\times\mathbb{R}^n}{
\eta^{\omega}\ge R^{k\omega v^*\top}y+S^{k\omega v^*},\;y\in\mY^{\prime}}.
\end{displaymath}
Let $\{\lambda^{k\omega*},\zeta^{k\omega*},\sigma^{k\omega*}_v\;\forall v\in\mathcal{L}(y^k,\omega)\}$ 
be the optimal solution of the linear program \eqref{opt:disj-cut}.
Then the inequality $\eta^{\omega}\ge\lambda^{k\omega*\top}y+\zeta^{k\omega*}$ 
from \eqref{eqn:cut-omega} must be valid for the point $(Q(y^k,\omega),y^k)$,
i.e., we have $Q(y^k,\omega)\ge\lambda^{k\omega*\top}y^k+\zeta^{k\omega*}$. 
We show that $Q(y^k,\omega)=\lambda^{k\omega*\top}y^k+\zeta^{k\omega*}$ by contradiction.
Assume that $Q(y^k,\omega)>\lambda^{k\omega*\top}y^k+\zeta^{k\omega*}$, 
then there exists an $\epsilon>0$ such that 
$Q(y^k,\omega)-\epsilon\ge\lambda^{k\omega*\top}y^k+\zeta^{k\omega*}$. 
It implies that the point $(Q(y^k,\omega)-\epsilon,y^k)$ is a point in the set 
$\textrm{conv}(\cup_{v\in\mathcal{L}(y^k,\omega)}E^{k\omega}_v)$. 
Therefore, there exist a subset $\mathcal{S}$ of $\mathcal{L}(y^k,\omega)$,
a set of points $\{(\eta^{\prime}_v,y^{\prime}_v)\}_{v\in\mathcal{S}}$, 
and a set of coefficients $\{\alpha_v\}_{v\in\mathcal{S}}$ satisfying that
$(\eta^{\prime}_v,y^{\prime}_v)\in E^{k\omega}_v$ for all $v\in\mathcal{S}$, and
the following convex-combination equations:
\begin{equation}\label{eqn:Q_x_conv_comb}
\begin{aligned}
&Q(y^k,\omega)-\epsilon = \sum_{v\in\mathcal{S}}\alpha_v\eta^{\prime}_v, 
\quad y^k=\sum_{v\in\mathcal{S}}\alpha_vy^{\prime}_v,  
\quad \sum_{v\in\mathcal{S}}\alpha_v=1, 
\quad \alpha_v> 0 \quad \forall v\in\mathcal{S}.
\end{aligned}
\end{equation}
Since $y^{\prime}_v\in\mY^{\prime}$ for all $v\in\mathcal{S}$ and
$y^k$ is an extreme point of $\mY^{\prime}$, it follows that $y^{\prime}_{v}=y^k$
for all $v\in\mathcal{S}$ in \eqref{eqn:Q_x_conv_comb}. 
The equations in \eqref{eqn:Q_x_conv_comb}
further imply that there exists a node $v_0\in\mathcal{S}$
satisfying $\eta^{\prime}_{v_0}\le Q(y^k,\omega)-\epsilon$ and 
$\eta^{\prime}_{v_0}\ge R^{k\omega v_0\top}y^{\prime}_{v_0}+S^{k\omega v_0}
=R^{k\omega v_0\top}y^k+S^{k\omega v_0}$. Let $x^{v_0}$ be the optimal solution 
to the node-$v_0$ relaxation of $\textrm{Sub}(y^k,\omega)$ in the branch-and-cut method. 
Using strong duality, we have 
\begin{displaymath}
Q(y^k,\omega)-\epsilon\ge\eta^{\prime}_{v_0}\ge R^{k\omega v_0\top}y^k+S^{k\omega v_0}
=q^{\omega\top}x^{v_0}\ge Q(y^k,\omega),
\end{displaymath}
which leads to a contradiction. Therefore, we must have $Q(y^k,\omega)=\lambda^{k\omega*\top}y^k+\zeta^{k\omega*}$.
\end{proof}
\begin{remark}
We now discuss an implication of Lemma~\ref{lem:Q=Ry+Z}.
Suppose $y^k$ is the optimal solution of the master problem at the $k^{\textrm{th}}$ iteration.
If $y^k$ is not the optimal solution of \eqref{opt:DR-TSS-MISOCP},
Lemma~\ref{lem:Q=Ry+Z} implies that in a later iteration $k^{\prime}$ ($k^{\prime}>k$),
we must have $y^{k^{\prime}}\neq y^k$. Otherwise, the master problem at Iteration $k^{\prime}$ 
(a relaxation of \eqref{opt:DR-TSS-MISOCP}) will return a lower bound that is no smaller
than the upper bound value provided in Iteration $k$. It means that if $y^k$ is not an optimal solution 
of \eqref{opt:DR-TSS-MISOCP}, it will be cut off by the scenario constraint 
$\eta^{\omega}\ge \lambda^{k\omega*\top}y+\zeta^{k\omega*}$. 
Formal analysis to explain this implication is given in the proof of Theorem~\ref{thm:finite-convg}.
\end{remark}

\begin{theorem}\label{thm:finite-convg} 
Suppose there exists an integer $N$ such that after $N$ iterations in Algorithm~\ref{alg:decomp-BC},
each second-stage problem \emph{Sub($y^k,\omega$)} with $k>N$  is solved to optimality using a branch-and-cut algorithm. 
Then Algorithm~\ref{alg:decomp-BC} returns an optimal solution of \eqref{opt:DR-TSS-MISOCP} after finitely many iterations.
\end{theorem}
\begin{proof}
Let  the first-stage problem at iteration $k$ be denoted by Master-$k$.
For any $k>N$, Lemma~\ref{lem:Q=Ry+Z} implies that 
$Q(y^k,\omega)=\lambda^{k\omega\top}y^k+\zeta^{k\omega}$.
Therefore, we have
\begin{equation}\label{eqn:maxEQ}
G(y^k):=\underset{P\in\mathcal{P}}{\textrm{max}}\;\E_P[Q(y^k,\omega)]=
\sum_{\omega\in\Omega} p^k_{\omega}\lambda^{k\omega\top}y^k + \sum_{\omega\in\Omega}p^k_{\omega}\zeta^{k\omega}.
\end{equation}
Based on the mechanism of the algorithm, it is clear that if the algorithm terminates in finitely many iterations,
it returns an optimal solution. 
We only need to show that the algorithm must terminate in finitely many iterations.
Assume that the algorithm does not terminate in finitely many iterations.
Then it must generate an infinite sequence of first-stage solutions $\{y^k\}^{\infty}_{k=1}$. Consider the subsequence 
$\{y^k\}^{\infty}_{k=N+1}$. There must exist two first-stage solutions $y^{k_1}$ and $y^{k_2}$ in this subsequence
satisfying that $y^{k_1}=y^{k_2}$, with $k_1<k_2$. At the end of iteration $k_1$ the upper bound $U^{k_1}$ satisfies
\begin{equation}\label{eqn:U}
U^{k_1}= c^{\top}y^{k_1}+\sum_{\omega\in\Omega}p^{k_1}_{\omega}Q(y^{k_1},\omega)=
c^{\top}y^{k_1} + G(y^{k_1}),
\end{equation}
where \eqref{eqn:maxEQ} is used to obtain the last equation.
The optimal value of Master-$k_2$ gives the lower bound $L^{k_2}=c^{\top}y^{k_2}+\eta^{k_2}$. 
Since $k_2>k_1$,  Master-$k_2$ has the following constraint:
\begin{equation}\label{eqn:eta}
\eta \ge \sum_{\omega\in\Omega} p^{k_1}_{\omega}\lambda^{k_1\omega\top}y + \sum_{\omega\in\Omega}p^{k_1}_{\omega}\zeta^{k_1\omega}.
\end{equation}
Therefore, we conclude that
\begin{equation}\label{eqn:Lk2}
\begin{aligned}
L^{k_2}&=c^{\top}y^{k_2}+\eta^{k_2} \\
&\ge c^{\top}y^{k_2}+
\sum_{\omega\in\Omega} p^{k_1}_{\omega}\lambda^{k_1\omega\top}y^{k_2} + \sum_{\omega\in\Omega}p^{k_1}_{\omega}\zeta^{k_1\omega} \\
&=c^{\top}y^{k_1}+\sum_{\omega\in\Omega} p^{k_1}_{\omega}\lambda^{k_1\omega\top}y^{k_1} + \sum_{\omega\in\Omega}p^{k_1}_{\omega}\zeta^{k_1\omega} \\
&=c^{\top}y^{k_1}+G(y^{k_1})=U^{k_1}= U^{k_2},
\end{aligned}
\end{equation}
where we use the fact that $y^{k_1}=y^{k_2}$, and  inequalities \eqref{eqn:U}--\eqref{eqn:eta} to obtain \eqref{eqn:Lk2}.
Hence, we have no optimality gap at the solution $y^{k_1}$, and the algorithm should have terminated at 
iteration $k_1$. 
\end{proof}
We remark that Algorithm~\ref{alg:decomp-BC} terminates after generating a finitely many linear
inequalities, even though the set $\mP$ is any convex set and the second-stage problems are second-order-cone
programs, both of which allow to have infinitely many extreme point solutions.
\begin{algorithm}
	\caption{A decomposition branch-and-cut algorithm for solving \eqref{opt:DR-TSS-MISOCP}.}
	\label{alg:decomp-BC}
	\begin{algorithmic}
	\State{Initialization: $L\gets -\infty$, $U\gets \infty$, $k\gets 1$.}
	\While{$U-L>\epsilon$}
		\State{Solve the first-stage problem \eqref{opt:masterprob} to optimality. Let $(\eta^k,y^k)$ be the optimal solution.}
		\State{Update the lower bound as $L\gets\textrm{current optimal value of the master problem}$.}
		\State{Set the current best solution as $y^*\gets y^k$.}
		\For{$\omega\in\Omega$:}
			\State{Solve every second-stage problem $\textrm{Sub}(y^k,\omega)$ for $\omega\in\Omega$ to a suitable accuracy}
			\State{$\quad$using the branch-and-cut method.}
			\State{Let $x^{k\omega}$ be the best feasible solution identified in solving $\textrm{Sub}(y^k,\omega)$.}
			\State{Derive the scenario based optimality constraint \eqref{eqn:cut-omega} by solving \eqref{opt:disj-cut}.}
		\EndFor
		\State{Solve the optimization \eqref{opt:worst-distr} to get the current worst-case probability distribution $\bs{p}^k$.}
		\State{Aggregate the inequalities \eqref{eqn:cut-omega} to get the inequality \eqref{eqn:agg-cut}.}
		\State{Add the aggregated optimality constraint \eqref{eqn:agg-cut} to the first-stage problem \eqref{opt:masterprob}.}
		\State{Update the upper bound as $U\gets \textrm{min}\{U,\; c^{\top}y^k+\sum_{\omega\in\Omega}p^k_{\omega}q^{\omega\top}x^{k\omega}\}$}.
		\State{$k\gets k+1$.}
	\EndWhile
	\State{Return $y^*$.}
	\end{algorithmic}
\end{algorithm}

\section{Generalization of the Decomposition Method for DR-TSS Mixed-integer Conic Programs}
\label{sec:DR-TSS-MICP}
The decomposition method from the previous section can be generalized for solving DR-TSS mixed-integer convex conic programs (DR-TSS-MICP). These models allow more general cone representable problems such as those having  SDP cones, and exponential cones \cite{chandrasekaran2016} by making an appropriate choice of $\mK$. The main difference in the algorithm presented in this section from that in the previous section is that here we use the convex conic dual for each leaf node relaxation in the second-stage problem of DR-TSS-MICP, which is a generalization of the SOCP dual for DR-TSS-MISOCP.
We make the following assumption on \eqref{opt:DR-TSS-MICP} which is
counterpart to Assumption~\ref{ass:feasible} for \eqref{opt:DR-TSS-MICP}:
\begin{assumption}\label{ass:feasb-micp}
\begin{itemize}
	\item[\emph{(a)}] The \eqref{opt:DR-TSS-MICP} problem has a complete recourse, i.e., for any feasible solution $y$
			  of the first stage problem, all second-stage problems are feasible.
	\item[\emph{(b)}] For any feasible first-stage solution $y$ and for any $x^1\in\mathbb{Z}^{l_1}$ satisfying $x^1_i\in[z^{L\omega}_i,\;z^{U\omega}_i]$, 
			   there exists a $x^2\in\mathbb{R}^{l_2}$
			   such that the solution $x=[x^1;x^2]$ is strictly feasible to the second-stage problem \eqref{opt:micp-subprob}.
\end{itemize}
\end{assumption}
For a given first-stage solution $y^k$, we solve the second-stage problem \eqref{opt:micp-subprob}
using the branch-and-cut method. 
We use $\textrm{CSub}(y^k,\omega)$ to denote the second-stage problem \eqref{opt:micp-subprob} 
associated with $y^k$ and $\omega$, and use $\mL(y^k,\omega)$ to denote the set of leaf nodes
when solving $\textrm{CSub}(y^k,\omega)$.
The relaxation of $\textrm{CSub}(y^k,\omega)$ at a node $v\in\mL(y^k,\omega)$ is given by:
\begin{equation}\label{opt:micp-node}
\begin{aligned}
&\underset{x^{\omega}}{\textrm{min}}\;\; q^{\omega\top}x^{\omega}  & \qquad \textrm{dual variables}\\
&\textrm{ s.t. } W^{\omega} x^{\omega} \ge r^{\omega}-T^{\omega} y^k,   &\qquad \gamma^{k\omega v} \\
&\qquad x^{\omega}\ge z^{L\omega}_v,\; x^{\omega}\le z^{U\omega}_v,  &\qquad \tau^{k\omega v}_L,\;\tau^{k\omega v}_U \\
&\qquad x^{\omega}\in \mK_2,  &\qquad s^{\omega} \in \mK^*_2,
\end{aligned}
\end{equation}
where $\mK^*_2$ is the dual cone of $\mK_2$ defined as \newline
$\mK^*_2=\Set*{s\in\mathbb{R}^{l_1+l_2}}{\braket{s}{x}\ge 0\;\forall x\in\mK_2}$.
We dualize \eqref{opt:micp-node} using the conic duality theory \cite{shapiro2001_conic-lp}.
The dual problem of \eqref{opt:micp-node} is formulated as:
\begin{equation}\label{opt:micp-node-dual}
\begin{aligned}
&\textrm{max}\;\; \gamma^{k\omega v\top}(r^{\omega}-T^{\omega}y^k)
+\tau^{k\omega v\top}_Lz^{L\omega}_v-\tau^{k\omega v\top}_Uz^{U\omega}_v, \\
&\textrm{ s.t. }\; q^{\omega}-s^{\omega}-W^{\omega\top}\gamma^{k\omega v}+\tau^{k\omega v}_U-\tau^{k\omega v}_L\in\mK^*_2, \\
&\qquad\; \gamma^{k\omega v},\;\tau^{k\omega v},\;\tau^{k\omega v}_U\ge 0,\; s^{\omega}\in\mK^*_2.
\end{aligned}
\end{equation}
Moreover, the optimal objective value of \eqref{opt:micp-node-dual} 
is equal to the optimal value of \eqref{opt:micp-node} (see Appendix~\ref{app:strong-dual}).
The following proposition is the counterpart of 
Proposition~\ref{prop:dual-ineq} for the case of general cone.
\begin{proposition}\label{prop:conic-dual-ineq}
Let $y^*$ be an optimal solution of \eqref{opt:DR-TSS-MICP}, and let $x^{\omega*}$ be an optimal second-stage
solution of $\emph{CSub}(y^*,\omega)$. Let $\gamma^{k\omega v}$, $\tau^{k\omega v}$, $\tau^{k\omega v}_U$, and $s$
be the dual variable values determined by solving the relaxation problem of  $\emph{CSub}(y^k,\omega)$ at node $v$.
If $z^{L\omega}_v\le x^{\omega*}\le z^{U\omega}_v$, then the following inequality holds:
\begin{equation}\label{eqn:Q>Ry+S}
Q(y^*,\omega)\ge -\gamma^{k\omega v\top}T^{\omega}y^* + \gamma^{k\omega v\top}r^{\omega}
+\tau^{k\omega v\top}_Lz^{L\omega}_v-\tau^{k\omega v\top}_Uz^{U\omega}_v:=\psi(y^*).
\end{equation}  
\end{proposition} 
\begin{proof}
We reorganize terms in the objective and constraints of the primal and dual problems, 
and then make use of primal, dual feasibility and strong duality. 
Specifically, we have
\begin{equation}
\begin{aligned}
&\psi(y^*)= -\gamma^{k\omega v\top}T^{\omega}y^* + \gamma^{k\omega v\top}r^{\omega}
+\tau^{k\omega v\top}_Lz^{L\omega}_v-\tau^{k\omega v\top}_Uz^{U\omega}_v \\
&=\gamma^{k\omega v\top}(r^{\omega}-T^{\omega}y^*)
+\tau^{k\omega v\top}_Lz^{L\omega}_v-\tau^{k\omega v\top}_Uz^{U\omega}_v \\
&\le \gamma^{k\omega v\top}(r^{\omega}-T^{\omega}y^*)
+\tau^{k\omega v\top}_Lz^{L\omega}_v-\tau^{k\omega v\top}_Uz^{U\omega}_v \\
&\quad + \braket{q^{\omega}-s-W^{\omega\top}\gamma^{k\omega v}+\tau^{k\omega v}_U-\tau^{k\omega v}_L}{x^{\omega*}} \\
&=\braket{q^{\omega}}{x^{\omega*}}-\gamma^{k\omega v\top}(W^{\omega}x^{\omega*}-r^{\omega}+T^{\omega}y^*)-\braket{s}{x^{\omega*}} \\
&\quad -\tau^{k\omega v\top}_L(x^{\omega*}-z^{L\omega}_v)-\tau^{k\omega v\top}_U(z^{U\omega}_v-x^{\omega*}) \\
&\le \braket{q^{\omega}}{x^{\omega*}}=Q(y^*,\omega),
\end{aligned}
\end{equation}
where the first and second inequalities make use of the primal, dual feasibility condition. 
\end{proof}
Due to Proposition~\ref{prop:conic-dual-ineq}, 
as in \eqref{eqn:Ekvw}, we construct an epigraph for each node $v\in\mL(y^k,\omega)$ as follows:
\begin{equation}
E^{k\omega}_v=\Set*{(\eta^{\omega},y)\in\mathbb{R}\times\mathbb{R}^n}{\eta^{\omega}\ge R^{k\omega v\top}y+S^{k\omega v},\;y\in\mY^{\prime}},
\end{equation}
where $R^{k\omega v}$ and $S^{k\omega v}$ are the coefficient vector of $y^*$ and the constant in
the right-side expression of \eqref{eqn:Q>Ry+S}, i.e., 
\begin{equation}\label{eqn:RS}
\begin{aligned}
&R^{k\omega v}=-T^{\omega\top}\gamma^{k\omega v}  \\
&S^{k\omega v}=\gamma^{k\omega v\top}r^{\omega}
+\tau^{k\omega v\top}_Lz^{L\omega}_v-\tau^{k\omega v\top}_Uz^{U\omega}_v. 
\end{aligned}
\end{equation}
The set $\mY^{\prime}$ is a polytope defined
as $\mY^{\prime}=\Set*{y\in\mathbb{R}^n}{Fy\ge a,\;0\le y_j\le 1\;\;\forall j\in[n]}$.
The scenario based optimality constraint in this case is the same as \eqref{eqn:cut-omega},
with coefficients being the optimal solution to the linear program \eqref{opt:worst-distr},
where the coefficients $R^{k\omega v}$, $S^{k\omega v}$  are given in \eqref{eqn:RS}.
The scenario based optimality constraint in this case is 
\begin{equation}\label{eqn:cp-disj-cut}
\eta^{\omega}\ge\lambda^{k\omega\top}y+\zeta^{k\omega},
\end{equation}
where the coefficients $\lambda^{k\omega},\zeta^{k\omega}$ are the optimal solution of
the linear program: 
\begin{equation}\label{opt:cp-disj-cut}
\begin{aligned}
&\textrm{max}\;\;\lambda^{\top}y^k+\zeta \\
&\textrm{ s.t. }\; \{\lambda,\zeta,\sigma_{v},\gamma_v\;\forall v\in\mathcal{L}(y^k,\omega)\}\in V^{k\omega},
\end{aligned}
\end{equation}
where $V^{k\omega}$ in \eqref{opt:cp-disj-cut} is defined as follows:
\begin{equation}\label{eqn:cp-Vk}
V^{k\omega}=\Set*{\begin{aligned}&\lambda,\zeta,\sigma_{v},\gamma_v\\ 
&\forall v\in\mathcal{L}(x^k,\omega)\end{aligned}}
{
\begin{aligned}
&\lambda-R^{k\omega v}+F^{\top}\sigma_v-\gamma_v \le 0 \quad \forall v\in\mathcal{L}(x^k,\omega), \\
&\zeta-\sigma^{\top}_{v}a-S^{k\omega v}+\gamma^{\top}_v\bs{1} \le 0 \quad \forall v\in\mathcal{L}(x^k,\omega), \\
&\sigma_{v}\ge 0,\;\gamma_v\ge 0
\end{aligned}
},
\end{equation}
and the coefficients $R^{k\omega v}$ and $S^{k\omega v}$ are given in \eqref{eqn:RS}. The worst-case probability distribution based on the current first-stage solution $y^k$
is determined via the optimization problem \eqref{opt:worst-distr}.
The aggregated constraint in this case is 
\begin{equation}\label{eqn:cp-agg-cut}
\eta\ge \sum_{\omega\in\Omega}p_{\omega}\lambda^{k\omega\top}y+\zeta^{k\omega},
\end{equation}
where the worst-case probability distribution is given by the optimal solution of:
\begin{equation}\label{opt:cp-worst-case-distr}
\underset{\bs{p}\in\mathcal{P}}{\textrm{max}}\;\sum_{\omega\in\Omega}p_{\omega}(\lambda^{k\omega\top}y^k+\zeta^{k\omega}).
\end{equation}
Algorithm~\ref{alg:Cone-decomp-BC} to solve \eqref{opt:DR-TSS-MICP} is analogous to Algorithm~\ref{alg:decomp-BC}.
The master problem of \eqref{opt:DR-TSS-MICP} at Iteration $k$ is as follows:
\begin{equation}\label{opt:cone-master}
\begin{aligned}
&\underset{y}{\textrm{min}}\; c^{\top}y+\eta  \\
&\textrm{ s.t. } \eta\ge h^{l}-(f^{l})^{\top}y \quad \forall l\in\{1,\ldots,k-1\}, \\
&\qquad Fy\ge a, \\
&\qquad y\in\mK_1\cap\{0,1\}^n.
\end{aligned}
\end{equation}
Theorem~\ref{thm:cone-finite-convg} states that Algorithm~\ref{alg:Cone-decomp-BC} can solve 
\eqref{opt:DR-TSS-MICP} to optimality in finitely many iterations. 
A proof of Theorem~\ref{thm:cone-finite-convg} is analogous to that of Theorem~\ref{thm:finite-convg}.
\begin{theorem}\label{thm:cone-finite-convg}
Suppose there exists an integer $N$ such that after $N$ iterations in Algorithm~\ref{alg:decomp-BC},
each second-stage problem \emph{CSub($y^k,\omega$)} with $k>N$  is solved to optimality using a branch-and-cut algorithm. 
Then Algorithm~\ref{alg:decomp-BC} returns an optimal solution of \eqref{opt:DR-TSS-MICP} after finitely many iterations.
\end{theorem}

\begin{algorithm}
	\caption{A decomposition branch-and-cut algorithm for solving \eqref{opt:DR-TSS-MICP}.}
	\label{alg:Cone-decomp-BC}
	\begin{algorithmic}
	\State{Initialization: $L\gets -\infty$, $U\gets \infty$, $k\gets 1$.}
	\While{$U-L>\epsilon$}
		\State{Solve the first-stage problem \eqref{opt:cone-master} to optimality. Let $(\eta^k,y^k)$ be the optimal solution.}
		\State{Update the lower bound as $L\gets\textrm{current optimal value of the master problem}$.}
		\State{Set the current best solution as $y^*\gets y^k$.}
		\For{$\omega\in\Omega$}
			\State{Solve every second-stage problem $\textrm{CSub}(y^k,\omega)$ for $\omega\in\Omega$ to some accuracy}
			\State{$\quad$using the branch-and-cut method.}
			\State{Let $x^{k\omega}$ be the best feasible solution identified in solving $\textrm{CSub}(y^k,\omega)$.}
			\State{Derive the scenario based optimality constraint \eqref{eqn:cp-disj-cut} by solving \eqref{opt:cp-disj-cut}.}
		\EndFor
		\State{Solve the optimization \eqref{opt:cp-worst-case-distr} to obtain the current worst-case probability distribution $\bs{p}^k$.}
		\State{Aggregate the inequalities \eqref{eqn:cp-disj-cut} to obtain the inequality \eqref{eqn:cp-agg-cut}.}
		\State{Add the aggregated optimality constraint \eqref{eqn:cp-agg-cut} to the first-stage problem \eqref{opt:cone-master}.}
		\State{Update the upper bound as $U\gets \textrm{min}\{U,\; c^{\top}y^k+\sum_{\omega\in\Omega}p^k_{\omega}q^{\omega\top}x^{k\omega}\}$}.
		\State{$k\gets k+1$.}
	\EndWhile
	\State{Return $y^*$.}
	\end{algorithmic}
\end{algorithm}

\section{An Illustrative Numerical Example} \label{sec:smallExample}
We now provide a numerical example to illustrate the decomposition branch-and-cut algorithm developed in this paper.
Consider a \eqref{opt:DR-TSS-MISOCP} instance with four scenarios $\Omega=\{\omega_1,\omega_2,\omega_3,\omega_4\}$. 
The optimization problem is as follows:
\begin{equation}
\begin{aligned}
&\underset{y}{\textrm{min}}\;\; 10y_1+12y_2 + \underset{\bs{p}\in\mathcal{P}}{\textrm{max}}\; 
\big[ p_1Q(y,\omega_1)+p_2Q(y,\omega_2) +  p_3Q(y,\omega_3)+p_4Q(y,\omega_4)\big]  \\
&\textrm{ s.t. }\; y_1+y_2\ge 1, \qquad\; y_1,y_2\in\{0,1\}
\end{aligned}
\end{equation}
We let the ambiguity set be defined using the total-variance metric as follows:
\begin{equation}
\mathcal{P}=\Set*{p\in\mathbb{R}^2}{p_1+p_2+p_3+p_4=1,\;d_{TV}(p,p^0)\le 0.1},
\end{equation}
where $p^0=(1/4,1/4,1/4,1/4)$ is the nominal probability distribution, and the total-variance metric $d_{TV}$ 
for the problem instance is given by $d_{TV}(p,p^0)=\sum^4_{i=1}|p_i-1/4|$. 
We may use alternative definitions of $\mP$ such as a set defined using the Wasserstein metric, a moment
based set, or a $\phi$-divergence based set \cite{rahimian2019_dro-review}.
The second-stage problems are given as follows. 
\newline
Scenario $\omega_1$:
\begin{equation}
\begin{aligned}
Q(y,\omega_1)=\;&\underset{x}{\textrm{min}}\;\; 2x_1+x_2 \\
&\textrm{ s.t. } 
 \left\Vert
\begin{bmatrix}
x_1+0.5y_1 \\
x_2+0.5y_2
\end{bmatrix}  
 \right\Vert_2\le 0.5x_1+x_2+1, \\
&\qquad x_1+x_2\ge 0.5y_1+0.5y_2, \;  x_1\in\{0,1\},\; 0\le x_2\le 1.
\end{aligned}
\end{equation}
Scenario $\omega_2$:
\begin{equation}
\begin{aligned}
Q(y,\omega_2)=\;&\underset{x}{\textrm{min}}\;\; 1.5x_1+1.5x_2 \\
&\textrm{ s.t. } 
 \left\Vert
\begin{bmatrix}
x_1+0.5y_1 \\
x_2+0.5y_2
\end{bmatrix}  
 \right\Vert_2\le 0.5x_1+x_2+1, \\
&\qquad x_1+x_2\ge 0.5y_1+0.5y_2, \; x_1\in\{0,1\},\; 0\le x_2\le 1.
\end{aligned}
\end{equation}
Scenario $\omega_3$:
\begin{equation}
\begin{aligned}
Q(y,\omega_3)=\;&\underset{x}{\textrm{min}}\;\; 1.2x_1+1.5x_2 \\
&\textrm{ s.t. } 
 \left\Vert
\begin{bmatrix}
x_1+0.5y_1 \\
x_2+0.5y_2
\end{bmatrix}  
 \right\Vert_2\le 0.5x_1+x_2+1.5, \\
&\qquad x_1+x_2\ge 0.5y_1+0.5y_2, \; x_1\in\{0,1\},\; 0\le x_2\le 1.
\end{aligned}
\end{equation}
Scenario $\omega_4$:
\begin{equation}
\begin{aligned}
Q(y,\omega_4)=\;&\underset{x}{\textrm{min}}\;\; x_1+x_2 \\
&\textrm{ s.t. } 
 \left\Vert
\begin{bmatrix}
x_1+0.5y_1 \\
x_2+0.5y_2
\end{bmatrix}  
 \right\Vert_2\le 0.5x_1+1.5x_2+1, \\
&\qquad x_1+x_2\ge 0.5y_1+0.5y_2,\; x_1\in\{0,1\},\; 0\le x_2\le 1.
\end{aligned}
\end{equation}

Consider an initial first-stage solution $y^{0}=(1,1)$. The root relaxation of $\textrm{Sub}(y^{0},\omega_1)$
gives a feasible solution $x^{\omega_1*}=(0,1)$, and $Q(y^{0},\omega_1)=1$.
The branch-and-cut tree of $\textrm{Sub}(y^{0},\omega_1)$ contains only one 
leaf node (the root node): $\mathcal{L}(y^0,\omega_1)=\{v_{11}\}$.
The (scenario-node) constraint of scenario $\omega_1$ at node $v_{11}$ is given by:
\begin{equation}
    \eta^{\omega_1}\ge 0.7097y_1+0.7097y_2-0.4194. \label{eqn:c1}
\end{equation}
Since there is only one node in $\mathcal{L}(y^0,\omega_1)$,
taking the union of epigraphs is not needed for $\mathcal{L}(y^0,\omega_1)$.
Now consider scenario $\omega_2$. The root relaxation of $\textrm{Sub}(y^{0},\omega_2)$
gives a second-stage solution $x^{\omega_2}=(0.5361,0.4639)$. The branch-and-cut tree of 
$\textrm{Sub}(y^{0},\omega_2)$ contains two leaf nodes $\mathcal{L}(y^0,\omega_2)=\{v_{21},v_{22}\}$,
where the feasible subsets of the two nodes are:
\begin{displaymath} 
\begin{aligned}
& \mathcal{Z}^{\omega_2}_{v_{21}}=\Set*{x\in\mathbb{R}^2}{x_1=0,\;0\le x_2\le 1}; 
\quad \mathcal{Z}^{\omega_2}_{v_{22}}=\Set*{x\in\mathbb{R}^2}{x_1=1,\;0\le x_2\le 1}.
\end{aligned}
\end{displaymath}
By solving the node relaxation SOCP at $v_{21}$ and $v_{22}$, we obtain the following constraints for scenario $\omega_2$ nodes $v_{21}$ and $v_{22}$:
\begin{displaymath}
\begin{aligned}
v_{21}\;\textrm{constraint}:&\ \eta^{\omega_2}\ge 0.75y_1+0.75y_2 \\
v_{22}\;\textrm{constraint}:&\ \eta^{\omega_2}\ge 1.125y_1+0.4687y_2+0.0938.
\end{aligned}
\end{displaymath}
The recourse function value at scenario $\omega_2$ is $Q(y^0,\omega_2)=1.5$.
Solving disjunctive programming formulation \eqref{opt:disj-cut},  
we generate a valid scenario constraint:
\begin{equation}\label{eqn:c2}
\eta^{\omega_2}\ge 0.75y_1+0.75y_2. 
\end{equation}
The parameters in \eqref{opt:disj-cut} are given by $y^0=(1,1)$, $R^{0\omega_21}=(0.75,0.75)^{\top}$, 
$R^{0\omega_22}=(1.125,0.4687)^{\top}$, $S^{0\omega_21}=0$, $S^{0\omega_22}=0.0938$,
$F=(1,1)$ and $a=1$.
Note that this disjunctive constraint is the same as the constraint for node  $v_{21}$ corresponding to scenario $\omega_2$.
For scenario $\omega_3$, the root relaxation SOCP gives solution $x^{\omega_3*}=(1,0)$.
The recourse function value at scenario $\omega_3$ is $Q(y^0,\omega_3)=1.2$. 
The corresponding scenario constraint is given by:
\begin{equation}\label{eqn:c3}
\eta^{\omega_3}\ge 0.65y_1+0.65y_2-0.1. 
\end{equation}
For scenario $\omega_4$, the branch-and-cut tree of $\textrm{Sub}(y^0,\omega_4)$ contains
two leaf nodes: $\mathcal{L}(y^0,\omega^4)=\{v_{41},v_{42}\}$.
By solving the node relaxation SOCP at $v_{41}$ and $v_{42}$, we obtain the following scenario-node constraints:
\begin{displaymath}
\begin{aligned}
\omega_4, v_{41}\;\textrm{constraint}:&\quad \eta^{\omega_4}\ge 0.5y_1+ 0.5y_2 \\
\omega_4, v_{42}\;\textrm{constraint}:&\quad \eta^{\omega_4}\ge 0.4082y_1+0.155 y_2+0.5064.
\end{aligned}
\end{displaymath}
The recourse function value at scenario $\omega_4$ is $Q(y^0,\omega_4)=1$.
Solving disjunctive programming formulation \eqref{opt:disj-cut}, we generate a valid scenario constraint
\begin{equation}\label{eqn:c4}
\eta^{\omega_4}\ge 0.5y_1+0.5y_2. 
\end{equation}
The parameters in \eqref{opt:disj-cut} are given by $y^0=(1,1)$, $R^{0\omega_41}=(0.5,0.5)^{\top}$, 
$R^{0\omega_42}=(0.4082,0.155)^{\top}$, $S^{0\omega_41}=0$, $S^{0\omega_42}=0.5064$,
$F=(1,1)$ and $a=1$.
Note that this disjunctive constraint is the same as the  constraint for $\omega_4$ from node $v_{41}$.
The worst-case probability distribution is given by the following linear program:
\begin{equation}\label{opt:num-worst-prob}
\begin{aligned}
&\underset{p}{\textrm{max}}\;\;\sum^4_{i=1}p_iQ(y^0,\omega_i)  \\
&\textrm{ s.t. } \sum^4_{i=1}|p_i-1/4|\le 0.1, \quad \sum^4_{i=1}p_i=1,\quad p\in\mathbb{R}^4_+.
\end{aligned}
\end{equation}
The worst-case probability distribution is $p^*=(0.25, 0.3, 0.25, 0.2)$.
Aggregating constraints \eqref{eqn:c1}-\eqref{eqn:c4} using this worst-case probability distribution
$p^*$, we obtain the following aggregated constraint:
\begin{equation}\label{eqn:agg1}
\eta \ge 0.665 y_1+ 0.665y_2 -0.12985.
\end{equation}
The algorithm adds constraint \eqref{eqn:agg1} to the first-stage problem. In the next iteration 
the lower and upper bounds are $L=10.535$, $U=23.2$.
The lower bound is attained at $y^1=(1,0)$.  

We repeat the cut generation steps for the current first-stage solution $y^1=(1,0)$, and obtain the following aggregated constraint:
\begin{equation}
\eta\ge 0.60375y_1+0.6375y_2.
\end{equation}
Adding this constraint to the first-stage problem, we obtain an updated lower bound $L=10.6375$,
which is attained at $y^2=(1,0)$. The updated upper bound is $U=10.6375$.
Since $L=U$, the optimal solution is $y^*=(1,0)$.

\section{Numerical Experiments using a Distributionally Robust Facility Location Model}
\label{sec:num-exp}
We now discuss computational performance of the decomposition branch-and-cut algorithm developed in this paper to solve larger instances of DR-TSS-MISOCP. For this purpose we consider a reformulation of the utility robust facility location problem (RFL) investigated in \cite{luo2019-dro-fl}. In the RFL problem, we need to decide location of service centers from a given set of locations.  Customers gain a certain utility from the service center location decisions.
The objective is to maximize the expectation of total utilities gained by customers. 
The RFL model in \cite{luo2019-dro-fl} is studied for the case where the customer demand is deterministic. In the numerical experiments of this paper, 
we investigate a general version of the RFL model allowing for stochastic customer demand.
We assume that the demand has a finite support (finitely many scenarios), and we robustify 
the model by evaluating the expected total utility via the worst-case probability distribution over the scenarios. 
The RFL problem with stochastic demand subject to distributional ambiguity is given as follows:
\begin{equation}\label{opt:RFL}
\begin{aligned}
&\underset{\bs{y}}{\textrm{max}}\;\;\bs{c}^{\top}\bs{y}+\underset{P\in\mathcal{P}^{\Omega}}{\textrm{min}}\E_{P}[\mathcal{Q}(\bs{y},\omega)] \\
&\textrm{ s.t. } \sum_{j\in F} b_j y_j \le B,   \\
&\qquad y_j\in\{0,1\}\;\forall j\in F,
\end{aligned}
\tag{SD-RFL}
\end{equation}
where $\bs{y}$ is the location vector of facilities, $b_j$ is the cost of opening a facility at location $j\in F$,
$B$ is the budget, and the  the recourse function 
$\mathcal{Q}(\bs{y},\omega)$ represents the total utility of customers in scenario $\omega$,
and it is given by the following second-stage problem: 
\begin{equation}\label{opt:RFL-II-2}
\begin{aligned}
& \mathcal{Q}(\bs{y},\omega)\;=\; \textrm{max}\;\; \sum_{i\in S}\sum_{j\in F}\Uoij \\
& \textrm{ s.t. }\; \Uij{1} \le (\hbeta^{ij})^{\top}\bvij{1}-b^{ij}\|(\bs{A}^{ij})^{-1/2}\bvij{1}\|,\\
&\qquad \Uij{2} \le (\hbeta^{ij})^{\top}\bvij{2}-\sqrt{\gamma^{ij}_2}\|(\whb{\Sigma}^{ij})^{1/2}\bvij{2}\|, \\
&\qquad \Uoij = \Uij{1}+\Uij{2}, \quad \bvoij = \bvij{1}+\bvij{2}, \\
&\qquad v^{\omega ij1}_k\le R^{\omega ij}\soij,
   \quad v^{\omega ij2}_k\le R^{\omega ij}(1-\soij)    \\
&\qquad v^{\omega ij}_k\le R^{\omega ij}y_k,\quad v^{\omega ij}_k\le x^{\omega}_{ij},\quad v^{\omega ij}_k\ge x^{\omega}_{ij}-R^{\omega ij}(1-y_k),\\
& \qquad \forall i\in S,\;\forall j\in F,\;\forall k\in F, \\
&\qquad \bs{x}^{\omega}\in X^{\omega}(\bs{y}),\; s^{\omega ij}\in\{0,1\},\;  U^{\omega ij}, \Uij{1}, \Uij{2} \ge 0, \\
&\qquad \bs{v}^{\omega ij}, \bvij{1}, \bvij{2} \in\mathbb{R}^{|F|}_+ \quad \forall i\in S,\;\forall j\in F,
\end{aligned}
\tag{RSP}
\end{equation}
where the feasible set $X^{\omega}(\bs{y})$ is defined as 
\begin{equation}\label{eqn:X-set}
X^{\omega}(\bs{y}):=\Set*{x^{\omega}_{ij}\;\forall i\in S,\;\forall j\in F}
{
\begin{array}{ll}
\sum_{i\in S}x^{\omega}_{ij} \le C_jy_j & \forall j\in F, \\
\sum_{j\in F}x^{\omega}_{ij} \le D_i & \forall i\in S, \\
x^{\omega}_{ij}\ge 0 & \forall i\in S,\; \forall j\in F
\end{array}
}.
\end{equation}
The parameter $R^{\omega ij}$ is defined as $R^{\omega ij}=\textrm{min}\{D^{\omega}_i,C_j\}$, 
where $D^{\omega}_i$ is the demand from customer site $i$ in scenario $\omega$,
and $C_j$ is the capacity of facility at location $j$.
Note that \eqref{opt:RFL} is a DR-TSS-MISOCP problem. Here we have used the results in \cite{luo2019-dro-fl} for the deterministic (single scenario case) to give the more general formulation for the stochastic case. The two second order cone and the binary variables appearing in the model are due to reformulation of a worst-case utility defined over an ambiguity set. A detailed explanation of  the deterministic model formulation is provided in \cite{luo2019-dro-fl}. 

\subsection{Numerical Instance Generation}
We generated 18 instances of \eqref{opt:RFL} to test the computational performance of the decomposition algorithm. The instances are labeled as FL0, FL1, \ldots, FL17.
We now describe the numerical instance generation. 
The number of customer sites $|S|$ is given in the second column of Table~\ref{tab:cuts-500}. 
The customer sites are points  in a $15\times 15$ two-dimensional square. 
The two coordinates of each customer site are generated using a uniform random variable in the range $[0,15]^2$. 
Every customer site is also a candidate service center location, i.e., $F=S$. 
The parameters $\bs{c}$ that represent the extra gain in establishing service centers
in the \eqref{opt:RFL} model are set to zero in all the numerical instances. 
Therefore, the instances only consider the total expected utility gained by the customers, which is computed from the evaluation of the second stage problems. Since the optimal solution is determined primarily based on the solution of  the second stage problems, the models become considerably harder. The cost of establishing each service center is 1, i.e., $b_j=1$ for all $j\in F$ in \eqref{opt:RFL}. 
The total budget is given in the third column of Table~\ref{tab:cuts-500}.
For every $j\in F$, the capacity $C_j$ is generated  from the interval $[100,180]$ using a uniform distribution.
To define the parameters $\hat{\bs{\beta}}^{ij}$, we follow the approach in \cite{luo2019-dro-fl}. We first define an effective distance $L_0=5$,
and define an effective set $F_i$ of service centers for each $i\in S$ such that 
$F_i=\Set*{j\in F}{\|\bs{x}^j-\bs{x}^i\|_2\le L_0}$, where $\bs{x}^i$ is the coordinate vector of the customer site $i\in S$.  The parameters $\hat{\bs{\beta}}^{ij}$ are set as follows:
\begin{equation}
    \hat{\beta}^{ij}_k = \left\{\def\arraystretch{1.2}
			 \begin{array}{ll}  
			  10\times \left(1-\|\bs{x}^i-\bs{x}^j\|_2/L_0\right)  & \qquad \textrm{if } j\in F_i\textrm{ and }k=j  \\
			  1-\|\bs{x}^i-\bs{x}^k\|_2/L_0 & \qquad \textrm{if } j\in F_i\textrm{ and }k\neq j  \\
			  0 & \qquad \textrm{if } j\in F\setminus F_i,\;\forall k\in F.
			  \end{array}
			  \right.
\end{equation}
Thus, the parameters reflect inverse proportionality to utility with respect to distance. 
The covariance matrix $\widehat{\bs{\Sigma}}^{ij}$ (for all $i\in S,\;j\in F$) is set to be $\widehat{\bs{\Sigma}}^{ij}=\bs{Q}^{ij\top}\bs{Q}^{ij}$,
where $\bs{Q}^{ij}$ is a $|F|\times |F|$ matrix with each entry randomly generated from $[0,1]$.
The matrix $\bs{A}^{ij}$ (for all $i\in S,\;j\in F$) is set to be $\bs{A}^{ij}=\bs{I}_{|F|}+0.3\;\bs{Q}^{ij\top}\bs{Q}^{ij}$,
where $\bs{I}_{|F|}$ is the $|F| \times |F|$ identity matrix. We set $\gamma^{ij}=0.2$ and $b^{ij}=0.2$ for all $i\in S$, $j\in F$.
In our test results we use the total variation (TV) distance \cite{villani2008} as the metric to measure 
the distance between two probability distributions on the scenario space. 
The ambiguity set $\mathcal{P}^{\Omega}$ in \eqref{opt:RFL} is defined as 
\begin{equation}\label{def:TV}
\mathcal{P}^{\Omega}=\Set*{\bs{p}}{\sum_{\omega\in\Omega}p_{\omega}=1,
						\;\sum_{\omega\in\Omega}|p_{\omega}-p^0_{\omega}|\le d_{TV}},
\end{equation}
where $\bs{p}^0$ is the nominal probability distribution which is defined as $p^0_{\omega}=1/|\Omega|$
in the experiments. We set the ambiguity distance $d_{TV}=0.1$ in the experiments.
For each scenario, the demand $D^{\omega}_i$ is generated randomly from the interval $[40,60]$
for all $i\in S$. Instance FL1 with 500 scenarios are labeled as 
FL1-500, and the labels are similar for other instances. 

\subsection{Cuts for Second Stage Models}
It is possible to strengthen the second stage models through generation of tangent inequalities. These tangent inequalities are described in \cite{luo2019-dro-fl} for a single scenario (deterministic) case. However, they are also applicable for the second stage model as the developed inequalities are independent of the first stage decisions. Properties of the cuts and numerical methods used to generate these tangent inequalities for the model under consideration
are given in Section~4 of \cite{luo2019-dro-fl}. We added these inequalities for $\omega\in\Omega$, $i\in S$ and $j\in F$ at the $t^{\textrm{th}}$main iteration of
the decomposition algorithm. Specifically, the added cuts take the form:
\begin{equation}\label{eqn:cuts}
U^{\omega ij}\le \bs{\tau}^{\omega ijt\top}\bs{v}^{\omega ij}+\rho^{\omega ijt},
\end{equation}
where $\bs{\tau}^{\omega ijt}$ and $\rho^{\omega it}$ are coefficients of a cut.

\subsection{Experience with an Extended Formulation of \eqref{opt:RFL}}
\label{sec:comp-ext-form}
We first tested whether an off-the-shelf solver can handle an extensive form reformulation of \eqref{opt:RFL}.  In this test, we set $d_{TV}=0$, i.e., no demand ambiguity is assumed. 
The extended reformulation used in this test is given in Appendix~\ref{app:ext-reform}. 
To test the effectiveness of the extended reformulation approach we solved small instances: FL0 with 100, 500 and 1000 scenarios. 
We also solved these instances using the proposed decomposition algorithm.
We give the best objective value, the optimality gap and the solution time for these approaches in Table~\ref{tab:extend-model}.  These numerical results show that the extensive formulation as well as the decomposition approach can solve FL0-100, FL0-500 and FL0-1000 instances to optimality. 
Although the extensive form reformulation of FL0 can be solved to optimality, the total time (loading time plus solution time) required by the 500 and 1000 scenario instances is more than ten times greater than that for the decomposition algorithm. We also note that the difference between the objective values of FL0 with 500 scenarios and 1000 scenarios is about 0.037\%, suggesting that the objective value has nearly converged for problems with 500 scenarios. 

The extensive form formulation of FL1-500 has $2.12\times 10^7$ continuous variables,
$2\times 10^5$ binary variables, $3.26\times 10^7$ rows, $2.14\times 10^7$ columns, and $4\times 10^5$ quadratic constraints. The Gurobi solver takes about 50 minutes to load the model and it runs out of memory (4 GB) when solving the root relaxation problem. 

\begin{table}
\centering
\begin{threeparttable}
{\scriptsize
\caption{Numerical performance of solving a small instance of \eqref{opt:RFL} using the extended reformulation,
and the decomposition branch-and-cut algorithm (DBC).}\label{tab:extend-model}
\begin{tabular}{cc|cccc|ccc}
\hline\hline
& & \multicolumn{4}{c|}{Extended Formulation} &  \multicolumn{3}{c}{DBC} \\
\hline
ID	&	$|\Omega|$	&	Obj	&	Gap(\%)	&	Load-T(s)	&	Sol-T(s)  &	Obj	&	Gap(\%)	&	Sol-T(s)	\\
\hline
FL0	&	100	&	7094.28	&	0.00	&	189	&	194		&	7094.28	&	0.00	&	78	\\
FL0	&	500	&	7121.65	&	0.00	&	952	&	1381	&	7121.65	&	0.00	&	157	\\
FL0	&	1000	&	7119.02	&	0.00	&	2034	&	4495	&	7119.02	&	0.00	&	361	\\
\hline
\end{tabular}
\begin{tablenotes}
	\item The total variation distance is set to be 0.
\end{tablenotes}
}
\end{threeparttable}
\end{table}

\subsection{Experience with the Decomposition Algorithm for Solving \eqref{opt:RFL}}
\label{sec:comp-add-cuts}
We used the decomposition algorithm to solve 18  \eqref{opt:RFL} instances of the facility location problem models as stochastic programs ($d_{TV}=0$), and their distributionally robust counterparts ($d_{TV}=0.1$). The results for $|\Omega|=500$ and $d_{TV}=0$ are given in Table~\ref{tab:cuts-500-dtv-0} and those with $|\Omega|=500$ and $d_{TV}=0.1$ are given in Table~\ref{tab:cuts-500}. 
The time limit is set to 24 hours and 60 cores are used for each instance when solving the second stage programs. In Table~\ref{tab:cuts-500}, the `Init. LB' and `Init. gap' columns give the initial lower bound and optimality gap at the beginning of the algorithm. The initial first-stage solution is determined by solving a RFL model with deterministic demand. The deterministic demand is taken to be expected demand, where the expectation is calculated from the generated data. The initial lower bound is the \eqref{opt:RFL} objective evaluated at the first-stage solution for this deterministic model. 
The `Obj' and `final gap' columns give the best objective value and the final optimality gap when the 24-hour time limit is reached. The optimality gap is defined as $(UB-LB)/UB\times 100\%$, where $LB$ and $UB$ are lower and upper bounds of the optimal value.
The `Iters' column is the number of main iterations in the decomposition algorithm.
The `masT' and `scenT' columns give the percentage of time spent on the master problem and scenario problems, respectively.
The number of cuts generated is given in the column `Cuts', 
which is averaged over the iterations and scenarios.

Tables~\ref{tab:cuts-500-dtv-0} and ~\ref{tab:cuts-500} show that, except for the FL0 instances, the computational time spend in solving  scenario problems is more than 90\%. It is seen that instance FL0-500 is solved to optimality, and
for the other 17 instances optimality gap remains when 24 hour time limit is reached. For instances FL1-500, FL2-500 and FL4-500 the final optimality gap is relatively small. For the stochastic programming instances of these problems, as indicated in Table~\ref{tab:cuts-500-dtv-0}, the gap is $0.4\%$, $1.5\%$ and $0.4\%$, respectively. The gap for these instances is approximately 0.5\%, 1.6\% and 0.5\% for the distributional robust case (see Table~\ref{tab:cuts-500}). This is a significant improvement to the gap at the initial solution. In both tables we observe that the objective value of the best known solution is improved at the first significant digit. For the results reported in Table~\ref{tab:cuts-500-dtv-0} the improvement in the objective value ranges from 1\% to 10\% (average 5.5\%), where as for the results reported in Table~\ref{tab:cuts-500} the improvement ranges from 1\% to 15\% (average 8.6\%). In all cases there is a very significant reduction in the optimality gap known at the initial solution. For the stochastic programming problems the optimality gap reduces from 67\% on the average to approximately 11\%. For the distributional robust counterparts this gap reduces from the inital gap of approximately 70\% to 11.8\%. The fact that initial solutions improve significantly suggests that the solutions obtained by ignoring randomness are sub-optimal, and the extent of sub-optimality increases with ambiguity in the demand distribution. 

When comparing results in Tables ~\ref{tab:cuts-500-dtv-0} and \ref{tab:cuts-500} we find that the statistics on algorithmic performance for solving the stochastic programming problem and its distributionally robust counterpart are similar. Importantly, the number of master iterations taken for the stochastic programming instances in Table~\ref{tab:cuts-500-dtv-0}, and the distributionally robust instances in Table~\ref{tab:cuts-500} are similar. On average 226 iterations were taken for the stochastic programming models in Table~\ref{tab:cuts-500-dtv-0} and 228 iterations were  taken for the distributionally robust models in Table~\ref{tab:cuts-500} within the computational time limit. This suggests that the problem complexity is not increasing in the distributional robustness framework when compared with the stochastic programming model. Note that the step of identifying the worst-case probability distribution over scenarios is not needed
for the case $d_{TV}=0$. We also observe that the model difficulty increases with the budget $B$ on the number of facilities that can be opened. In comparison with the models with $B=5$, the models with a larger value of $B$ have significantly greater optimality gap when terminating with 24-hour time limit. 

We also conducted numerical experiments for the first 10 instances (FL1-FL10) 
with 1000 scenarios under similar computational settings and using $d_{TV}=0.1$. The results are given in Table~\ref{tab:cuts-1000} of Appendix~\ref{app:num-res}. For all of these instances we found that the best solution identified within the time limit is the same as that for the 500-scenario instances.  The relative difference between the best objective values in the 500-scenario 
and 1000-scenario instances at termination is in the range of 0.017\% to 0.118\%. This suggests that a numerical convergence in the distributional robust model is achieved, as the sample average approximation becomes sufficiently accurate, when the number of scenarios is of moderate size.

\begin{table}
\centering
\begin{threeparttable}
{
\tiny
\caption{\footnotesize Numerical results for solving \eqref{opt:RFL} using the DBC algorithm. 
 Every instance has 500 scenarios and the total variation distance is set to be 0. 
The column `Cuts' gives the average number of cuts generated, 
where the average is taken over scenarios and iterations.
The values in columns `Init. Gap', `Gap', `masT' and `senT' are in percentage. 
}\label{tab:cuts-500-dtv-0}
\center
\begin{tabular}{ccccccccccc}
\hline\hline
ID	&	$|S|$	&	$B$	&	Init. LB 	&	Init. Gap 	&	Obj	&	Gap	&	Iters	&	masT	&	scenT	&	Cuts	\\
\hline
FL0-500	 &	10	&	5	&	7117.2	&	58.8	 &	7282.3	&	0.0	&	126	&	23.8	 &	76.2	 &	63	\\
FL1-500  &	20	&	5	&	5386.6	&	63.1	 &	5438.7	&	0.4	&	921	&	7.8	&	92.2	 &	141	\\
FL2-500	 &	40	&	5	&	5257.6	&	59.3	 &	5457.8	&	1.5	&	323	&	7.9	&	92.1	 &	264	\\
FL3-500	 &	40	&	10	&	6530.6	&	66.2	 &	7214.5	&	20.9 	&	155	&	7.8	&	92.2	 &	262	\\
FL4-500  &	60	&	5	&	5291.0	&	64.8	 &	5599.1	&	0.4	&	143	&	7.5	&	92.5	 &	390	\\
FL5-500	 &	60	&	10	&	8111.5	&	71.6	 &	8482.0	&	7.9	&	294	&	6.7	&	93.3	 &	473	\\
FL6-500 &	80	&	5	&	4691.8	&	70.6	 &	5018.3	&	6.7	&	254	&	6.2	&	93.8	 &	530	\\
FL7-500	 &	80	&	10	&	7940.8	&	69.7	 &	8769.8	&	16.7	 &	210	&	7.2	&	92.8	 &	503	\\
FL8-500 	&	100	 &	5	&	5402.9	&	68.4 	&	5552.4	&	8.2	&	204	&	7.2	 &	92.8	 &	643	\\
FL9-500	 &	100	&	10	&	9175.4	&	64.5	 &	9942.8	&	17.6	 &	169	&	6.8	&	93.2	 &	800	\\
FL10-500	&	200	&	5	&	5195.0	&	58.6	 &	5626.7	&	8.5	&	182	&	7.0	&	93.0 	&	691	\\
FL11-500	&	200	&	10	&	8364.2	&	67.5	 &	9050.4	&	21.7	 &	137	&	7.2	&	92.8	 &	1142	\\
FL12-500	&	300	&	5	&	5285.9	&	66.3	 &	5655.8	&	8.5	 &	187	&	6.5	&	93.5	 &	889	\\
FL13-500	&	300	&	10	&	8623.5	&	62.2	 &	9036.8	&	18.4	 &	147	&	5.8	&	94.2	 &	1012	\\
FL14-500	&	400	&	5	&	5074.5	&	74.7	 &	5526.3	&	8.0	 &	145	&	6.7	&	93.3	 &	758	\\
FL15-500	&	400	&	10	&	8466.4	&	74.7	 &	8952.5	&	23.3	 &	158	&	7.6	&	92.4	 &	863	\\
FL16-500	&	500	&	5	&	5499.0	&	68.7	 &	5983.1	&	9.3	 &	160	&	7.2	&	92.8	 &	1007	\\
FL17-500	&	500	&	10	&	8551.7	&	76.2	 &	9080.2	&	21.6	 &	147	&	7.3	&	92.7	 &	1245	\\
\hline
Average &  &  &  & 67.0  &  & 11.1  & 226  & 8.0 & 92.0 & 649 \\
\hline
\end{tabular}
\begin{tablenotes}
	\item Instance FL0-500 is solved to optimality in 158 seconds using 60 cores. 
	\item Instances FL1-500 to FL17-500 are solved using 60 cores with 24-hour time limit.
\end{tablenotes}
}
\end{threeparttable}
\end{table}

\begin{table}
\centering
\begin{threeparttable}
{\tiny
\caption{\footnotesize Numerical results for solving \eqref{opt:RFL} using the DBC algorithm.
Every instance has 500 scenarios and the total variation distance is set to be 0.1. 
 }\label{tab:cuts-500}
\center
\begin{tabular}{ccccccccccc}
\hline\hline
ID	&	$|S|$	&	$B$	&	Init. LB 	&	Init. Gap 	&	Obj	&	Gap 	&	Iters	&	masT	&	scenT	&	Cuts	\\
\hline
FL0-500  &10  &5  &6988.1  &55.9  &7091.28  &0  &138  &23.2  &76.8  &68  \\ 
FL1-500  &20  &5  &5324.6  &64.1  &5400.3  &0.5  &998  &7.3  &92.7  &133  \\ 
FL2-500  &40  &5  &5002.3  &67.9  &5378.7  &1.6  &289  &7.7  &92.3  &261  \\ 
FL3-500  &40  &10  &6315.5  &69.8  &7153.7  &20.8  &140  &7.9  &92.1  &259  \\  
FL4-500  &60  &5  &5134.7  &68.9  &5496.4  &0.5  &141  &7.5  &92.5  &363  \\ 
FL5-500  &60  &10  &7255.1  &71.5  &8326.5  &9.5  &269  &6.8  &93.2  &453  \\ 
FL6-500  &80  &5  &4129  &78.4  &4926.6  &7.2  &253  &6.5  &93.5  &483  \\ 
FL7-500  &80  &10  &7650.2  &70.7  &8555  &18.5  &214  &7.3  &92.7  &511  \\ 
FL8-500  &100  &5  &4765.3  &69.4  &5519.6  &8.4  &228  &7.7  &92.3  &687  \\ 
FL9-500  &100  &10  &8931.8  &69.4  &9682.6  &18.1  &197  &7  &93  &760  \\ 
FL10-500  &200  &5  &4845.2  &63.2  &5472.3  &8.2  &177  &7.1  &92.9  &632  \\ 
FL11-500  &200  &10  &8439.1  &71.1  &8992  &23.6  &146  &7  &93  &1217  \\ 
FL12-500  &300  &5  &5169.2  &65.2  &5575.1  &8.5  &179  &6.5  &93.5  &856  \\ 
FL13-500  &300  &10  &8435.7  &70.3  &8974.8  &19.7  &166  &6.2  &93.8  &983  \\  
FL14-500  &400  &5  &4855.7  &73.2  &5489.3  &9.4  &142  &6.8  &93.2  &792  \\ 
FL15-500  &400  &10  &8339.5  &76.4  &8832.3  &23.6  &139  &7.3  &92.7  &941  \\ 
FL16-500  &500  &5  &5344.2  &76.3  &5814.7  &9.8  &143  &6.9  &93.1  &922  \\ 
FL17-500  &500  &10  &8450.7  &80.2  &8939.2  &24.2  &137  &7.6  &92.4  &1203  \\ 
\hline
Average &  &  &  & 70.1  &   & 11.8  & 228  & 8.0  & 92.0  & 640   \\
\hline
\end{tabular}
\begin{tablenotes}
	\item Instance FL0-500 is solved to optimality in 174 seconds using 60 cores. 
	\item Instances FL1-500 to FL17-500 are solved using 60 cores with 24-hour time limit.
\end{tablenotes}
}
\end{threeparttable}
\end{table}

\section{Concluding Remarks} \label{sec:conclusions}
The decomposition branch-and-cut algorithm developed in this paper is a general purpose algorithm.  It was used to solve a DR-TSS-MISOCP reformulation of a stochastic programming service center location problems as well as its distributionally-robust counterpart. Results show that the decomposition algorithm achieves significant improvement in the solution time when compared to an extensive form formulation of the stochastic programming model. Achieving a reasonable optimality gap was only possible when considering the decomposition algorithm. For the test instances the numerical results suggest that the distributionally robust counterpart does not increase the model complexity. Despite significant improvements through the algorithmic development of this paper and the use of cuts to strengthen the second stage formulation, the test-case model remains hard and optimality gap remains when terminating with a 24-hour time limit. Additional strengthening of the second-stage problems by adding valid inequality constraints may further improve the computational performance of the algorithm on the test instances studied here. Identification of such formulation strengthening constraints and evaluation of their practical value is problem dependent, and it should be considered when solving instances of models from specific applications admitting mixed-integer conic second stage formulations. 

\begin{acknowledgements}
This research was supported by the Office of Naval Research grant N00014-18-1-2097-P00001.
\end{acknowledgements}

\bibliographystyle{spmpsci_unsrt}
\bibliography{reference-database-08-06-2019}

\appendix
\section{Strong Duality for Conic Linear Programming}
\label{app:strong-dual}
We consider the following conic linear program:
\begin{equation}\label{opt:CLP2}
\begin{aligned}
&\underset{x}{\textrm{min}}\;\;\langle c,x\rangle \\
&\textrm{ s.t. }A^1x=b^1,  \\
&\qquad A^2x\ge b^2, \\
&\qquad x\in \mK.
\end{aligned}
\tag{P}
\end{equation}
Note that \eqref{opt:CLP2} is a general formulation of the node relaxation second-stage problem considered in this paper.
The dual of \eqref{opt:CLP2} is:
\begin{equation}\label{opt:CLP2-dual}
\begin{aligned}
&\underset{\mu,\lambda,y}{\textrm{max}}\;\;\mu^Tb^1+\lambda^Tb^2  \\
&\textrm{ s.t. } c-A^{1\top}\mu-A^{2\top}\lambda-y\in\mK^*, \\
&\qquad y\in\mK^*,\;\;\lambda\ge \bs{0}.
\end{aligned}
\tag{D}
\end{equation}
The following theorem (Theorem~\ref{thm:strong-dual}) on strong duality is based on Theorem~2.5.4 of \cite{strong-conic-dual}.
\begin{theorem}[Strong Duality]\label{thm:strong-dual}
If the conic linear program \eqref{opt:CLP2} is feasible and has finite optimal value $\gamma$, 
and there exists an interior point $\tilde{x}\in int(\mK)$ satisfying $A^1\tilde{x}=b^1$,
$A^2\tilde{x}>b^2$, then the dual problem of \eqref{opt:CLP2} is feasible and has finite optimal value $\beta$
which is equal to $\gamma$. 
\end{theorem}
Theorem~\ref{thm:strong-dual} implies that the strong duality for the conic linear program \eqref{opt:CLP2} holds
 \eqref{opt:CLP2} has a non-empty relative interior.

\section{The DR-TSS-MISOCP Reformulation of Utility Robust Facility Location Problem}
\label{app:RFL}

\section{Extended Formulation}
\label{app:ext-reform}
The extended formulation of \eqref{opt:RFL} with $d_{TV}=0$ is given as follows:
\begin{equation}
\begin{aligned}
&\textrm{max}\;\bs{c}^{\top}\bs{y}+\sum_{\omega\in\Omega}\sum_{i\in S}\sum_{j\in F}p^0_{\omega}U^{\omega ij} \\
&\textrm{ s.t. } \sum_{j\in F}b_jy_j\le B,  \\
&\qquad \textrm{all constraints from }\eqref{opt:RFL-II-2}\textrm{ for all }\omega\in\Omega, \\ 
\end{aligned}
\tag{RFL-E}
\end{equation}
where $\bs{p}^0:=\{p^0_{\omega}:\;\omega\in\Omega\}$ is the nominal probability distribution over all scenarios.

\section{Additional Numerical Results}
\label{app:num-res}
\begin{table}[H]
\centering
\begin{threeparttable}
{
\tiny
\caption{\footnotesize Numerical results for solving instances with 1000 scenarios and the total variation distance set to be 0.1. 
The column `Diff(\%)' gives the relative absolute difference in the best objective value
from the 500-scenario problems.}\label{tab:cuts-1000}
\center
\begin{tabular}{ccccccccccc}
\hline\hline
ID	&	$|S|$,$B$	&	Init. LB 	&	Init. Gap 	&	Obj	&	Gap 	&	Iters	&	masT	&	scenT	&	Cuts & Diff 		\\
\hline
FL0-1000    &    10,  5    &    6969.7  &  53.1     &    7090.07   &  0.0  & 145    &   23.0   &  77.0  &   70  & 0.017  \\
FL1-1000	&	20,	5	&	5228.8	&	64.5	&	5398.8	&	0.6	&	687	&	5.8	&	94.2	&	137	&	0.028	\\
FL2-1000	&	40,	5	&	5020.3	&	71.5	&	5382.7	&	2.0	&	223	&	5.4	&	94.6	&	247	&	0.074	\\
FL3-1000	&	40,	10	&	6318.2	&	66.7	&	7148.9	&	27.8	&	106	&	7.2	&	92.8	&	274	&	0.067	\\
FL4-1000	&	60,	5	&	5340.9	&	66.7	&	5502.9	&	0.6	&	98	&	6.9	&	93.1	&	388	&	0.118	\\
FL5-1000	&	60,	10	&	7260.1	&	71.0	&	8319.3	&	14.0	&	182	&	3.9	&	96.1	&	402	&	0.086	\\
FL6-1000	&	80,	5	&	4190.4	&	84.9	&	4928.6	&	10.7	&	90	&	3.5	&	96.5	&	512	&	0.041	\\
FL7-1000	&	80,	10	&	7667.4	&	68.0	&	8558.7	&	21.4	&	80	&	5.8	&	94.2	&	523	&	0.043	\\
\hline
\end{tabular}
\begin{tablenotes}
	\item Instance FL0-1000 is solved to optimality in 392 seconds using 60 cores. 
	\item Instances FL1-1000 to FL7-1000 are solved using 60 cores with 24-hour time limit.
\end{tablenotes}
}
\end{threeparttable}
\end{table}
\end{document}